\documentclass[reqno, 11pt, a4paper]{amsart}
\usepackage{amsmath, amssymb}  
\usepackage{amsthm}
\usepackage{graphicx, color}
\usepackage{mathtools}
\usepackage[truedimen,top=3truecm, bottom=3truecm, left=2.5truecm, right=3.0truecm, includefoot]{geometry} 
\begin{document}


\title[Unbalanced spatial-segregation limit for two-species exclusion processes]{Spatial-segregation limit for exclusion processes with two components under unbalanced reaction}
\author[K. Hayashi]{Kohei Hayashi}
\address{Graduate School of Mathematical Sciences, The University of Tokyo, Komaba, Tokyo 153-8914, Japan.}
\email{kohei@ms.u-tokyo.ac.jp}
\maketitle

\newtheorem{definition}{Definition}[section]
\newtheorem{theorem}{Theorem}[section] 
\newtheorem{lemma}[theorem]{Lemma}
\newtheorem{proposition}[theorem]{Proposition}
\newtheorem{remark}{Remark}[section]
\newtheorem{assumption}{Assumption}[section]

\makeatletter
\renewcommand{\theequation}{%
\thesection.\arabic{equation}}
\@addtoreset{equation}{section}
\makeatother

\newcounter{num}
\newcommand{\Rnum}[1]{\setcounter{num}{#1}\Roman{num}}

\begin{abstract}
We consider exclusion processes with two types of particles which compete strongly with each other. In particular, we focus on the case where one species does not diffuse at all and killing rates of two species are given by monomials with distinct exponents. We study limiting behavior of interfaces which appear by such a strong competition. Consequently, three kinds of limiting behavior of interfaces (vanishing, moving and immovable interfaces as in \cite{IMMN17}) are derived directly from our interacting particle system taking advantage of hydrodynamic limit procedure with singular limit for annihilation dynamics. 
\end{abstract}

\section{Introduction}
\label{sec:intro}
Spatial-segregation limit (or fast-reaction limit in some literatures) is a problem that discusses the limiting behavior of the solution of competitive reaction-diffusion system when its competition rate tends to infinity. The problem has been studied in PDE theory, which is concretely described as follows. Let $\Omega$ be a bounded domain in $\mathbb{R}^d$ with smooth boundary and let $u_K$ and $v_K$ be a pair of non-negative solution of competition-diffusion system
\[
\begin{cases}
\begin{aligned}
\partial_t u = d_1 \Delta u - K c_1 (u,v) u v  \\
\partial_t v = d_2 \Delta v - K c_2(u,v) u v 
\end{aligned}
\end{cases}
\]
in $[0,\infty) \times \Omega$ and study the limiting behavior of $u_K$ and $v_K$ as $K$ tends to infinity under some boundary condition for each case. Here $d_1$ and $d_2 $ are non-negative constants (diffusion coefficients) and $c_1$ and $c_2$ are non-negative functions on $\mathbb{R}_+^2$. When one of $d_1$ and $d_2$ is zero we call this \textit{one-phase} case, while the case when $d_1$ and $d_2 $ are both strictly positive is called \textit{two-phase} case as our convention. Moreover, we call a pair of reaction rates $c_1$ and $c_2$ is \textit{balanced} if there exists a positive constant $\kappa$ such that $c_1 = \kappa c_2 $ and otherwise we call it \textit{unbalanced}. Here we remark that $c_1$ and $c_2$ physically denote annihilation (or killing) rates since we only consider the case when $c_1$ and $c_2$ are non-negative.

An early study for spatial-segregation limit in PDE theory is found in \cite{HHP96} for one-phase case with balanced killing rates when the spatial dimension is one. For two-phase case with balanced reaction rates, \cite{DHMP99} considered in the Neumann boundary condition and \cite{CDHMN04} in inhomogeneous Dirichlet boundary can be found.

Recently spatial-segregation limit problem described as above has been studied from a microscopic viewpoint in \cite{DMFPV19}. They consider an exclusion process with two components where two types of particles diffuse with different constant rates and they strongly compete with each other. When a type-$1$ particle and a type-$2$ particle stay at the same site, they annihilate simultaneously with rate $K( N )$ which depends on the scaling parameter $N$ and diverges as $N$ tends to infinity. Then they proved that taking hydrodynamic limit procedure for this process limiting behavior of macroscopic density of each type of particles is determined by a two-phase Stefan problem as derived for example in \cite{DHMP99} in PDE context. In other words, they derived the time-evolution of limiting interface which asymptotically appears as competition rate tends to infinity directly from an interacting particle system. Namely, the spatial-segregation limit problem can be understood from a microscopic point of view in a special case where two species compete with common and simple competition rates. One natural problem is to consider the case when the competition rates are divergent but essentially different between two distinct species. 

When reaction rates $c_1$ and $c_2$ are unbalanced, for instance when one type of particles has strong effect of competence while that of the other kind of particles is comparatively weak, limiting behavior of the solution of the corresponding reaction-diffusion system as $K$ tends to infinity is far from well understood in PDE context but there are a few results considering this unbalanced case. In \cite{IMMN17}, they study the case $c_1$ and $c_2$ are monomials with different exponents. They consider a reaction-diffusion system
\begin{equation}
\label{IMMN17 eq}
\begin{cases}
\begin{aligned}
&\partial_t u = \Delta u - K u^{m_1} v^{m_2} \\
&\partial_t v =               - K u^{m_3} v^{m_4}
\end{aligned}
\end{cases}
\end{equation}
with Neumann boundary condition focusing on following four cases: 
\[
\begin{aligned}
\text{Case \Rnum{1} } : \quad & m_1 >3, m_2 = m_3 = m_4 = 1, \\
\text{Case \Rnum{2} } : \quad & m_2 \geq 1, m_1 = m_3 = m_4 = 1, \\
\text{Case \Rnum{3} } : \quad & m_3 >1, m_1 = m_2 = m_4 = 1, \\
\text{Case \Rnum{4} } : \quad & 1 \leq m_4 < 2, m_1 = m_2 = m_3 = 1 .\\ 
\end{aligned}
\]
Then, they proved that there are three kinds of limiting behavior of the asymptotic interfaces: vanishing, moving and immovable interfaces. For Case \Rnum{1}, $v_K (t, \cdot) $ converges to $0$ for every $t > 0$ and $u_K$ converges to a solution of the heat equation on the whole domain $\Omega$ as $K$ tends to infinity. Thus the liming interface disappears instantaneously in this case. Though some intuitive arguments for corresponding dynamical system (see \cite{IMMN17} for detail) support a conjecture which says the interface vanishes also when $2 < m_1 \leq 3$ but this is not proved because of some technical reasons and we only consider the case for $m_1 > 3$ also in this article. For Case \Rnum{2} and \Rnum{4}, a transformation of $v_K$ enables us to get another competition-diffusion system with common reaction rates up to constant. This case is already studied in \cite{HHP96} or \cite{HHP00} and one gets the limiting interface governed by a one-phase Stefan problem. For Case \Rnum{3}, a limiting interface appears but it does not move at all. In this case, $u$ evolves according to the heat equation on the fixed domain with Dirichlet boundary condition and $v$ does not change it values in its domain (this domain is the complement of the domain where the time evolution of $u$ takes place). Our aim is to understand this result by conducting the scaling limit of interacting particle systems where the exponents of reaction terms are restricted to be positive integer. For this reason, though \cite{IMMN17} also studies Case\Rnum{4} (moving interface) with its exponent $1 \leq m_4 < 2$ in \cite{IMMN17}, this can not be derived from our interacting particle system except for the trivial case $m_4 =1$ (this case is contained in Case\Rnum{2}).

In this paper, we extend the microscopic model in \cite{DMFPV19} to consider a fast-reaction limit problem for Glauber-Kawasaki dynamics with unbalanced reaction rates. In \cite{DMFPV19}, they considered the case when $d_1,d_2 > 0$ and $c_1 \equiv c_2$ and derived a two-phase Stefan problem as a system of hydrodynamic limit equations. On the other hand, we in this paper consider the case when reaction rates are unbalanced (namely $c_1\neq \kappa c_2$ for any $\kappa$). For the first step to treat hydrodynamic limit problem with unbalanced reaction rates, we consider the case in \cite{IMMN17} where the corresponding PDE problem is discussed. Concretely, we consider a simple exclusion process with annihilation dynamics where annihilation rates of two species are given by monomials and further assume one type of particles does not diffuse at all (namely we assume $d_2= 0$ as \cite{IMMN17}). After some careful calculations, we can show that three types of limiting interfaces as in \cite{IMMN17} are derived through the hydrodynamic limit procedure for this process. One reason for considering the case $d_2 = 0$ is that it makes the problem technically simple to prove the PDE part, though this lack of Kawasaki dynamics for type-$2$ particles makes the probabilistic part more difficult. For Case \Rnum{1}, the assumption $d_2 = 0$ makes the second equation of (\ref{IMMN17 eq}) an ODE for $v$ provided $u$ is considered to be a given function, which can be solved explicitly. Therefore, the reaction-diffusion system (\ref{IMMN17 eq}) can be reviewed as a single equation so that the comparison principle becomes applicable. This plays a technically essential role in Case\Rnum{1}. On the other hand, for Case\Rnum{3}, $d_2 = 0$ is an essential assumption for the immovable interface to be deduced. In Case\Rnum{3}, if type-$2$ particles diffuse (namely $d_2$ is positive), then the system becomes equivalent to Case\Rnum{2} by ignoring magnitude of diffusion coefficients and thus the limiting interface does move as a solution of a one-phase Stefan problem. Hence the condition $d_2 = 0$ is not only a technical but also a phenomenologically essential assumption which makes limiting behavior of interfaces rich in variety.

Here we summarize what we prove as main theorems (Theorems \ref{main  thm1}, \ref{main thm2} and \ref{main thm3}) in this paper. Our main theorems state that three kinds of limiting behavior of asymptotic interfaces considered in \cite{IMMN17} is derived directly from interacting particle systems corresponding to each cases through the hydrodynamic limit procedure. The proof of the main theorems is divided into two parts: the probabilistic part and the PDE part. In the first part of this paper, we prove the probabilistic part of the hydrodynamic limit theorems by means of the relative entropy method introduced by H.-T. Yau in \cite{Yau91}. In that machinery, one needs to show that the probability law of an undergoing process and another probability measure parametrized by macroscopic quantities which are determined by some partial differential equation(s) are sufficiently close in terms of relative entropy. In our cases, we take a reference measure whose weight parameters satisfy a semi-discretized system (that is, discretized only for spatial variables so that the system becomes to be a system of ODEs) and study limiting behavior of this semi-discretized system as a deterministic problem after we proved the probabilistic part. We call the latter part PDE part and such a deterministic limiting procedure is conducted in semi-discretized settings. In other words, we treat the limit when both the scaling parameter and the reaction rate tends to infinity, while in \cite{IMMN17} they consider continuous reaction-diffusion systems and taking limit only for the reaction rate. In this paper, we first show the probabilistic part for our dynamics with general reaction terms, namely the case with $d_2 = 0$ and reaction rates $c_1$, $c_2$ are general, but $c_2$ is assumed to depend only on the configuration of type-$1$ particles for a technical reason, non-negative polynomials of configurations of each kind of particles. Then we study limiting behavior of solutions of the semi-discretized reaction-diffusion system only for specific reaction rates which macroscopically corresponds to the system of hydrodynamics limit equations (\ref{IMMN17 eq}). To conduct such a procedure, we actually have to put a restriction which ensures $c_2$ depends only on configuration of type-$1$ particles to prove particularly the probabilistic part, though it covers all cases considered in \cite{IMMN17}. We take a product Bernoulli measure for both types of particles as a reference measure which is parametrized by macroscopic densities of type-$1$ and $2$ particles. This is very natural for the dynamics for type-$1$ particles since product Bernoulli measure is stationary for Kawasaki dynamics and indeed Bernoulli measures parametrized by a spatially constant densities are a family of invariant measures for Kawasaki dynamics. For the Glauber-Kawasaki dynamics with single component where creation and annihilation rule is added, product Bernoulli measure with dynamical parameters which governed by a macroscopic equation (hydrodynamic limit equation) is known to be appropriate as reference measure in Yau's relative entropy method.

Finally we explain how this paper is organized. First, in Section \ref{sec:model} we give a precise description of our model and state main results and then we give the proof of main theorems in the forthcoming sections (Section \ref{sec:strategy} - \ref{sec:case3}). The proof is divided in two steps: probabilistic part and PDE part. In Section \ref{sec:strategy} we explain these steps which are needed to prove the main theorems in detail. Section \ref{sec:prob} is devoted to the proof of probabilistic part where we prove that the probability law of spatial density profiles of our microscopic dynamics is close to a reference measure which is dynamically parametrized by a solution of a semi-discretized reaction diffusion system. It should be noted here again that we use product Bernoulli measure as reference measure in our proof. In the last four sections (Section \ref{sec:PDE} - \ref{sec:case3}), we study limiting behavior of the semi-discretized system and show there are three regimes which derive three kinds of limiting behavior of interfaces: vanishing regime, moving regime and immovable regime. In Section \ref{sec:PDE}, we give a priori estimates involving the semi-discretized reaction-diffusion system. In Section \ref{case1}, we consider the case when the interface vanishes instantaneously (Case\Rnum{1} in \cite{IMMN17}). In this case, type-$2$ particles extincts and the density of type-$1$ particles evolves according to the heat equation on the whole domain at any positive time. In Section \ref{sec:case2}, we treat the case when reaction terms become the same up to constant through a change of variables (Case\Rnum{2} in \cite{IMMN17}). Then the limiting interface moves, which is governed by a one-phase Stefan problem. In Section \ref{sec:case3}, we investigate the case when the interface appears but it does not move at all (Case\Rnum{3} in \cite{IMMN17}).

\begin{remark}
Throughout this article, we use \textit{Proposition} to state results which are already known in other literatures.  
\end{remark}

\section{Our model and results}
\label{sec:model}
\subsection{Microscopic model}
Let $\mathbb{T}^d_N \cong \{1,..., N \}^d$ be the $d$-dimensional discrete torus and let $\mathcal{X}^2_N = \{ 0,1 \}^{\mathbb{T}^d_N} \times \{ 0,1 \}^{\mathbb{T}^d_N}$ be the configuration space of two kinds of particles. We denote an element $\eta= (\eta_1, \eta_2) \in \mathcal{X}^2_N$ with $\eta_i = \{ \eta_i (x) \}_{x \in \mathbb{T}^d_N}$ $(i=1,2) $. Here for each $i = 1,2$, an element $\eta_i$ represents the configuration of type-$i$ particles: $\eta_i(x)=1$ means there exists a type-$i$ particle on site $x \in \mathbb{T}^d_N$ and $\eta_i(x)=0$ means type-$i$ particle does not exist on site $x$. Next, we consider a kind of Glauber-Kawasaki processes which takes values on $\mathcal{X}^2_N$ as follows. We consider the simple exclusion process where only type-$1$ particles diffuse and its generator is given by
\[
L_K f(\eta_1, \eta_2) =
 \frac{1}{2} \sum_{ x,y \in \mathbb{T}^d_N , |x-y|=1 } \left [ f (\eta_1^{x,y}, \eta_2) - f(\eta_1, \eta_2) \right ]
\]
for  each function $f: \mathcal{X}^2_N \to \mathbb{R}$. Here, for each $\sigma \in \mathcal{X}_N$, $\sigma^{x,y}$ is the configuration after exchanging occupation variables on sites $x$ and $y$: 
\[
\sigma^{x,y}(z)= 
\begin{cases}
\sigma(x) \quad \text{if }  z=y, \\
\sigma(y) \quad \text{if } z=x, \\
\sigma(z) \quad \text{otherwise}. \\ 
\end{cases}
\]
On the other hand, the generator of the Glauber dynamics is given by $L_G = L_{1,G }+ L_{2, G}$ with 
\begin{align*}
L_{1, G} f (\eta_1, \eta_2) = 
 \sum_{x \in \mathbb{T}^d_N} c_{1,x} (\eta_1, \eta_2) \eta_1(x) \eta_2(x) \left[ f(\eta_1^x, \eta_2) -f(\eta_1, \eta_2) \right]  
\\
L_{2, G} f (\eta_1, \eta_2) = 
\sum_{x \in \mathbb{T}^d_N} c_{2,x} (\eta_1, \eta_2) \eta_1(x) \eta_2(x) \left[ f(\eta_1, \eta_2^x) -f(\eta_1, \eta_2) \right]
\end{align*}
for each function $f: \mathcal{X}^2_N \to \mathbb{R}$ where for each $\sigma \in \mathcal{X}_N$, $\sigma^x$ is the configuration after flipping the particle configuration on site $x$: 
\[
\sigma^x (z)= 
\begin{cases}
1-\sigma(x)  &\quad \text{if }  z=x, \\
\sigma(z) &\quad \text{if } z \neq x. \\
\end{cases}
\]
Here, $c_{i,x} (\eta_1, \eta_2)=c_i (\tau_x \eta_1, \tau_x \eta_2) $ for $i=1,2$ and $\{ \tau_x \}_{x \in \mathbb{Z}^d}$ are shifts acting on $\mathcal{X}_N$ as $\tau_x \sigma(\cdot) = \sigma (\cdot + x)$ for every $\sigma \in \mathcal{X}_N$. Throughout this article, we assume that both annihilation rates $c_i (\eta_1, \eta_2)(i=1,2)$ are non-negative and depend only on the particle configuration of finite number of sites which depend neither on $\eta_1(0)$ nor $\eta_2 (0) $ as follows:
\[
c_1 (\eta_1, \eta_2) = 
\sum_{\substack{ \Lambda_{ 1}, \Lambda_{ 2} \Subset \mathbb{Z}^d , 0 \not \in \Lambda_{ 1} \cup \Lambda_{ 2}   } } 
c_{1, \Lambda_{ 1}, \Lambda_{ 2} } \prod_{x \in \Lambda_{ 1} } \eta_1(x) \prod_{x \in \Lambda_{ 2} } \eta_2(x) 
\]
and 
\[
c_2 (\eta_1, \eta_2) \equiv c_2 (\eta_1) = 
\sum_{ \substack{ \Lambda \Subset \mathbb{Z}^d,  0 \not\in \Lambda  } } c_{2, \Lambda } \prod_{x \in \Lambda } \eta_1(x)  
\]
with some real constants $c_{1 ,\Lambda_1, \Lambda_2 }$ and $c_{ 2 , \Lambda } $ such that $c_1 $ and $c_2 $ stay non-negative and these constants are assumed to be zero except for finite numbers of them so that the above summation becomes finite. In this paper, we assume the reaction rate $c_2$ does not depend on configuration of type-$2$ particles. A technical reason for this assumption will be explained in the proof of the probabilistic part of main theorems given in Section 3 (see Remark \ref{c2 reason}). Moreover, for the sake of convenience described later we extend $c_1$ and $c_2$ as non-negative functionals on $ [0,1]^{\mathbb{T}^d_N} \times [0, 1] ^{\mathbb{T}^d_N} $ by
\[
c_1 (u, v) = 
\sum_{\substack{ \Lambda_{ 1}, \Lambda_{ 2} \Subset \mathbb{Z}^d \\ 0 \not \in \Lambda_{ 1} \cup \Lambda_{ 2}  } } 
c_{1 ,\Lambda_{ 1}, \Lambda_{ 2} } \prod_{x \in \Lambda_{ 1} } u(x) \prod_{x \in \Lambda_{ 2} } v(x) 
\]
and 
\[
c_2 (u,  v) \equiv c_2 ( u ) = 
\sum_{\Lambda \Subset \mathbb{Z}^d,  0 \not \in \Lambda  } c_{2 , \Lambda } \prod_{x \in \Lambda } u(x)  
\]
for $(u, v) \in [0,1]^{\mathbb{T}^d_N} \times [0, 1] ^{\mathbb{T}^d_N}$ with $u = \{u (x) \}_{x \in \mathbb{T}^d_N}$ and $v = \{v (x) \}_{x \in \mathbb{T}^d_N}$. In other words, $c_i ( u, v )$ is obtained by substituting $u (x)$ and $v (x)$ into $\eta_1(x)$ and $\eta_2 (x)$, respectively, in the definition of $c_1 (\eta_1, \eta_2)$ and $c_2 (\eta_1, \eta_2)$ for every $x \in \mathbb{T}^d_N$. 

For each $N \in \mathbb{N}$, let $\{ \eta^N_t = (\eta^N_{1,t}, \eta^N_{2,t}) \}_{t \geq 0}$ be the $\mathcal{X}^2_N$-valued Markov process generated by $L_N = N^2 L_K + K(N) L_G$ on some probability space $(\Omega_N, \mathcal{F}_N, \mathbb{P}^N_{\mu^N_0} )$. Here for a probability measure $\mu $ on the configuration space $\mathcal{X}^2_N$,  $\mathbb{P}^N_\mu $ is the probability measure under which the initial distribution of $\{ \eta^N_t \}_{t \geq 0 }$ is $\mu$ and we denote the expectation with respect to $\mathbb{P}^N_\mu$ by $\mathbb{E}^N_\mu[\cdot]$. An assumption for the initial distribution $\mu^N_0 $ will be described later (see assumption (A2)). Here $K=K(N)$ is a divergent parameter as $N$ tends to infinity, which corresponds to take so-called ``fast-reaction limit'' in PDE context. Define macroscopic empirical measures $\{ \pi^N_t = ( \pi^N_{1,t}, \pi^N_{2,t} ) \}_{t \geq 0}$ on the $d$-dimensional torus $\mathbb{T}^d \cong [0, 1 )^d$ by
\[
\pi^N_{i,t} (d\theta) := \frac{1}{N^d} \sum_{x \in \mathbb{T}^d_N} \eta^N_{i,t}(x) \delta_{ \frac{x}{N} } (d\theta), \quad i=1,2
\]
and hereafter we write $\langle \pi^N_{i,t} , \varphi \rangle:=\int_{\mathbb{T}^d} \varphi(\theta) \pi^N_{i,t} (d\theta)$ for any continuous function $\varphi$ on $\mathbb{T}^d$. Moreover, for any $\mathbb{R}^2$-valued continuous function $\varphi = (\varphi_1, \varphi_2)$, we denote its vector-valued integral by $\langle \pi^N_t , \varphi \rangle \coloneqq ( \langle \pi^N_{1,t} , \varphi_1 \rangle, \langle \pi^N_{2,t} , \varphi_2 \rangle ) $.

Our aim is to study the limiting behavior of spatial density profiles of both kinds of particles under dynamics such that diffusion of type-$1$ particles is speeded up by $N^2$ and two species compete with rate $K (N)$ which diverges as $N$ tends to infinity. Particularly, we will show that for special forms of reaction rates $c_i$ $(i=1,2)$ (that is, \textbf{Case 1}, \textbf{Case 2} and \textbf{Case 3} which are described later) there are three regimes of interface growth.

\subsection{Hydrodynamic limit}

To evaluate the difference between two probability measures, we use the relative entropy defined as follows. Let $\mu$ and $\nu$ be two probability measures on $\mathcal{X}^2_N$. We define the relative entropy of $\mu$ with respect to $\nu$ by
\begin{equation}
\label{rel ent}
H(\mu | \nu) \coloneqq \int_{\mathcal{X}^2_N} \frac{d\mu}{d\nu} \log{\frac{d\mu}{d\nu}} d\nu
\end{equation}
if $\mu$ is absolutely continuous with respect to $\nu$, while otherwise we define $H(\mu | \nu) := \infty$.

Next we summarize our assumptions on the initial distribution and state main theorems in this paper. 

\begin{description}
\item[\textmd{(A1)}] 
Let $u^N (0,x) = u^N_0 (x) $ and $v^N (0,x) = v^N_0 (x) $ be given and satisfy two bounds 
\[
\begin{aligned}
& e^{- C_1 K} \leq u^N (0, x)\leq C_2, \quad | \nabla^N u^N (0, x) | \leq C_0 K \\
& e^{- C_1 K} \leq v^N (0, x)\leq C_2, \quad | \nabla^N v^N (0, x) | \leq C_0 K 
\end{aligned}
\]
for every $x \in \mathbb{T}^d_N$ with $C_1 > 0$, $0 < C_2 < 1$ and $C_0 > 0$. 
Here $ \nabla^N $ is the discrete gradient, that is, for every $ u : \mathbb{T}^d_N \to \mathbb{R}$ we define $ \nabla^N u : \mathbb{T}^d_N \to \mathbb{R}^d$ by  $ \nabla^N u(x) = \left( N ( u(x + e_j ) - u (x) ) \right)_{j= 1,...,d} $. 
\item[\textmd{(A2)}]
We denote $\mu^N_0$ the distribution of $\eta^N_0 = (\eta^N_{1, 0} , \eta^N_{2, 0} )$ on $\mathcal{X}^2_N$ and let $\nu^N_0 $ be the product Bernoulli measure on $\mathcal{X}^2_N$ with mean $(u^N (0,\cdot), v^N (0, \cdot))$. We assume the relative entropy $H(\mu^N_0 | \nu^N_0 )$ defined by (\ref{rel ent}) satisfies $H ( \mu^N_0 | \nu^N_0 ) = O (N^{d-{\delta_0}})$ for some $\delta_0 >0$ as $N$ tends to infinity, that is, there exists a positive constant $M$ such that $H ( \mu^N_0 | \nu^N_0 ) \leq M N^{d-{\delta_0}}$ for sufficiently large $N$. 
\item[\textmd{(A3)$_\delta$}]
$K = K (N)$ satisfies $1 \leq K(N ) \leq \delta ( \log N )^{1/2} $ and $K (N) \to \infty$ as $N$ tends to infinity. 
\end{description}

For each $m\in \mathbb{N}$ and for some fixed sites $z_i \in \mathbb{Z}^d$ ($i=1,...,m-1$), we introduce the following three regimes which are special cases in our setting. 
\begin{description}
\item[Case 1.] $c_1(\eta_1,\eta_2)=\eta_1(z_1)\cdots \eta_1(z_{m-1})$, $c_2(\eta_1,\eta_2) \equiv 1$ with $m >3$. 

\item[Case 2.] $c_1(\eta_1,\eta_2)=\eta_2(z_1)\cdots \eta_2(z_{m-1})$, $c_2(\eta_1,\eta_2) \equiv 1$ with $m \ge 1$.

\item[Case 3.] $c_1(\eta_1,\eta_2) \equiv 1$, $c_2(\eta_1,\eta_2)=\eta_1(z_1)\cdots \eta_1(z_{m-1})$ with $m > 1$.
\end{description}
Here we suppose $c_1 \equiv 1$ when $m=1$ in \textbf{Case 2} and $c_2 \equiv 1$ when $m=1$ in \textbf{Case 3} by convention. For each regime, we have the following hydrodynamic limit result which describes the limiting behavior of interfaces between two particle territories.

In addition to the above assumptions (A1), (A2) and (A3) imposed for all three cases, we further introduce conditions (B1), (B2) and (B3) which are assumed for \textbf{Case 1}, \textbf{Case 2} and \textbf{Case 3}, respectively.

\begin{description}
\item[\textmd{(B1)}] 
There exist non-negative functions $u_0 \in C^4 (\mathbb{T}^d)$ and $v_0 \in C^\alpha (\mathbb{T}^d)$ for some $\alpha \in (0, 1) $ such that $u_0 \not\equiv 0 $, $u_0 v_0 \equiv 0$ and $v^N (0, x) = o (1/ K) $ for every $x \in \mathbb{T}^d_N$ satisfying $x / N \notin \text{supp} v_0 $ as $N $ tends to infinity, that is, we have \\
$\lim_{ N \to \infty } \sup_{ x/N \notin \text{supp} v_0 } K v^N (0, x) = 0$. Moreover, we assume 
\[
\lim_{ N \to \infty } \sup_{ x \in \mathbb{T}^d_N } K^2 | u^N (0, x) - u_0 (x/N) |  =0 , 
\quad \lim_{ N \to \infty} \sup_{ x \in \mathbb{T}^d_N } | v^N (0, x) - v_0 (x/ N) | = 0.
\]
\item[\textmd{(B2)}] 
There exist functions $u_0, v_0 \in L^2 (\mathbb{T}^d)$ such that $u_0 v_0 \equiv 0$ and functions $u^N ( 0 , \cdot)$ and $v^N ( 0 , \cdot)$ on $\mathbb{T}^d$ defined by (\ref{dsol}) satisfy for the index $m \ge 1 $ appearing in \textbf{Case 2} 
\[
u^N (0, \cdot) \rightharpoonup u_0, \, v^N (0, \cdot)^m \rightharpoonup v_0^m \text{ weakly in  } L^2 (\mathbb{T}^d)
\]
as $N$ tends to infinity. 
\item[\textmd{(B3)}] 
There exist functions $u_0 \in C (\mathbb{T}^d)$, $v_0 \in L^\infty (\mathbb{T}^d)$ and a positive constant $m_v$ satisfying $u_0,v_0 \not\equiv 0$, $u_0  v_0 \equiv 0$ and $v_0 \ge m_v$ in $\text{supp} (v_0) $ such that $u^N ( 0 , \cdot ) $ and $ v^N (0, \cdot ) $ on $\mathbb{T}^d$ defined by (\ref{dsol}) converge almost everywhere to $u_0$, $v_0$, respectively, as $N$ tends to infinity.
\end{description}

To state the main theorems in this article, we introduce the following notation. For functions $u$ and $v$ on $Q_T \coloneqq [0,T] \times \mathbb{T}^d$ such that $uv \equiv 0$ a.e. in $Q_T$, we define
\begin{equation}
\begin{aligned}
&\Omega^u(t) \coloneqq \{ \theta \in \mathbb{T}^d  |  u(t,\theta) >0 \},  \quad 
\Omega^v(t) \coloneqq \{ \theta \in \mathbb{T}^d  |  v(t,\theta) >0 \},  \\
&Q^u_T \coloneqq \bigcup_{0 \leq t \leq T} \{ t \} \times \Omega^u(t) , \quad 
Q^v_T \coloneqq \bigcup_{0 \leq t \leq T} \{ t \} \times \Omega^v(t) , \\
&\Gamma(t)  \coloneqq \Omega \backslash (\Omega^u(t) \cup \Omega^v(t)) , \quad 
\Gamma  \coloneqq  \bigcup_{0 \leq t \leq T}\{ t \} \times  \Gamma(t).
\end{aligned}
\end{equation}

First, when reaction rates are of \textbf{Case 1}, we can show that type-$1$ particles fill up the whole space and the limiting interface vanishes in an instant.

\begin{theorem}[Vanishing interface]
\label{main thm1}
Assume reaction rates $c_1$ and $c_2$ are of \textbf{Case 1}. Assume (A1), (A2), (A3)$_\delta$ and (B1) for some sufficiently small $\delta > 0$. Then for every $\varepsilon > 0$ and $\varphi = ( \varphi_1, \varphi_2 ) \in C^\infty( Q_T ; \mathbb{R}^2 )$ we have 
\[
\lim_{N\to \infty} \mathbb{P}^N_{\mu^N_0} \left( \left|\int_0^T \big( \langle\pi^N_t, \varphi(t, \cdot) \rangle - \langle ( u (t,\cdot) , 0 ) ,  \varphi (t, \cdot ) \rangle_{L^2(\mathbb{T}^d; \mathbb{R}^2 )} \big) dt \right|  > \varepsilon  \right) =0
\]
where the function $u$ is a classical solution of the heat equation on the whole domain with periodic boundary condition: 
\begin{equation}
\label{heat eq}
\begin{cases}
\partial_t u = \Delta u \\
u(0,\cdot)= u_0 (\cdot).
\end{cases}
\end{equation}
\end{theorem}

Figure \ref{macrocase1} shows an example of the evolution of limiting interface staring from some initial functions $u_0$ and $v_0$ for \textbf{Case 1}. For a typical pair of initial functions $(u_0, v_0)$, we can choose a semi-discretized initial functions $u^N_0 $ and $v^N_0$ satisfying the assumptions (A1) and (B1). For example, as shown in the left side of Figure \ref{macrocase1}, if $u_0$ and $v_0$ are smooth and bounded from above by $C_2$ appearing in the assumption (A1), then it suffices to take $u^N_0 ( x ) = \max \{ u_0 (x / N) , e^{- C_1 K } \} $ for every $x \in \mathbb{T}^d_N$. The initial function $v^N_0$ satisfying both assumptions (A1) and (B1) can be taken similarly.

\begin{figure}[h]
 \centering
 \includegraphics[width=14cm]{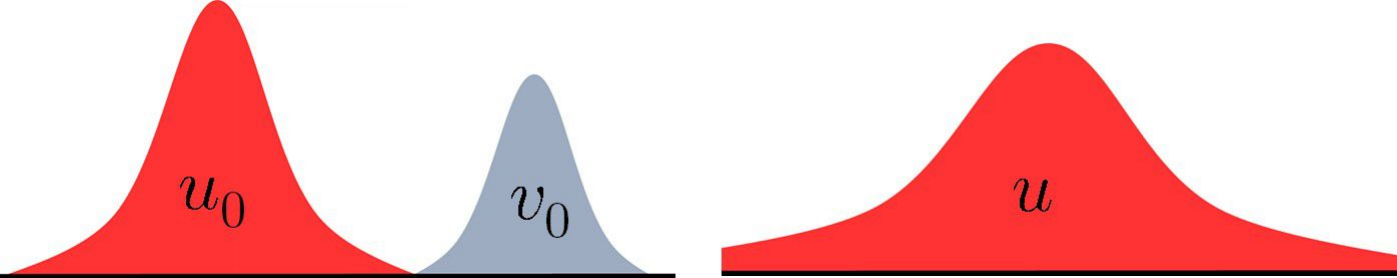}
 \caption{Interface behavior for \textbf{Case 1}. The left figure shows an example of initial functions $u_0$ and $v_0$, while the right one shows a typical situation at some positive time $t $. }
 \label{macrocase1}
\end{figure}

Next, when reaction rates are of \textbf{Case 2}, an interface between two territories appears and we can describe the motion of the interface by a one-phase Stefan problem with latent heat $v^m_0 |_{\Gamma(t)} /m$. Hereafter we define $w_+ \coloneqq \max \{  w , 0 \}$ and $w_- \coloneqq \max \{ -w, 0 \} $ for any $w$ which takes values in $\mathbb{R}$. 
To state the main result for \textbf{Case 2}, we introduce the notion of weak solution to a free boundary problem called the one-phase Stefan problem.

\begin{definition}
Let $w_0$ be a function in $L^\infty (\mathbb{T}^d)$. We call a function $w = w (t, \theta)$ on $Q_T = [0,T] \times \mathbb{T}^d$ a weak solution of the one-phase Stefan problem with initial function $w_0$ if 
\begin{description}
\item[\textmd{(1)}] $w \in L^\infty (Q_T), \, w_+ \in L^2 (0,T; H^1 (\mathbb{T}^d)) $. 
\item[\textmd{(2)}] For every $\varphi \in H^1 (Q_T) $ such that $\varphi (T, \cdot ) =0$, we have
\[
\int_0^T \int_{\mathbb{T}^d } ( w \partial_t \varphi - \nabla w_+ \cdot \nabla \varphi  ) d\theta dt = - \int_{\mathbb{T}^d} w_0  \varphi (0, \theta) d\theta .
\]
\end{description}
\end{definition}

\begin{theorem}[Moving interface]
\label{main thm2}
Assume reaction rates $c_1$ and $c_2$ are of \textbf{Case 2}. Assume (A1), (A2), (A3)$_\delta$ and (B2) with some $\delta > 0 $. Let $w $ be a unique weak solution of the one-phase Stefan problem with initial function $u_0 - v_0^m/m $ and let $u$ and $v$ be defined by $u = w_+ $ and $v^m = m w_-$ which satisfy $u (t, \theta ) v (t, \theta) = 0$ for every $\theta \in \mathbb{T}^d $. Then for every $\varepsilon > 0$ and $\varphi = ( \varphi_1 , \varphi_2 )  \in C^\infty( Q_T ; \mathbb{R}^2)$ we have
\[
\lim_{N\to \infty} \mathbb{P}^N_{\mu^N_0} \left( \left| \int_0^T \big( \langle\pi^N_t, \varphi (t, \cdot )  \rangle - \langle ( u (t,\cdot), v (t,\cdot) ),  \varphi (t, \cdot ) \rangle_{L^2(\mathbb{T}^d; \mathbb{R}^2 )} \big) dt  \right|    > \varepsilon  \right) =0.
\]
\end{theorem}

Uniqueness of the weak solution to the one-phase Stefan problem is proved in \cite{HHP00}. Moreover, we can show analogously in \cite{HHP96} that if the limiting interface $\Gamma (t) $ is smooth in $\mathbb{T}^d$ and $u$, $v$ are smooth on $\Omega^u (t) $, $\Omega^v (t) $ for every $t \in [0, T]$, respectively, then $u$ and $v$ satisfy the following free boundary problem in strong form: 
\begin{align}
\label{strong Stefan}
\begin{cases} 
\partial_t u = \Delta u &\text{ in }  Q^u_T \\
\dfrac{v_0^m }{m } V = - \dfrac{\partial u}{\partial n_\Gamma} &\text{ on } \Gamma \\
u=0 &\text{ on } \Gamma \\
u(0,\cdot) = u_0 (\cdot) &\text{ in } \Omega^u(0) \\
v \equiv v_0 & \text{ in } Q^v_T
\end{cases}
\end{align}
where $V$ is the normal velocity of the free boundary $\Gamma(t)$ and $n_\Gamma$ is the unit normal vector on $\Gamma(t)$ oriented from $\Omega^u(t)$ to $\Omega^v(t)$. When the above strong form holds, this system is called a one-phase Stefan problem with latent heat $v_0^m |_{\Gamma(t)} / m$. In this \textbf{Case 2}, there exists a nontrivial example of initial functions $u_0 $ and $v_0$ and their corresponding approximating sequences $u^N_0$ and $v^N_0$ satisfying the assumptions (A1) and (B2). For example, initial functions $u_0$ and $v_0$ are bounded from above by $C_2$ and suppose $u_0$ and $v_0$ are smooth on $\mathbb{T}^d$ and $\text{supp} (v_0)$, respectively, as shown in Figure \ref{macrocase2}. In this one-dimensional example, $\Omega^u (0) \cup \Omega^v (0)  = \mathbb{T}^d$ holds and $\Gamma (0) = \partial \Omega^u(0) = \partial \Omega^v (0) $ is a set consisting two points: one point is placed slightly right form the center and the other point is the identified endpoint in Figure \ref{macrocase2}. For these initial functions $u_0$ and $v_0$, we can choose approximating functions $u^N_0$ and $v^N_0$ by the same manner in \textbf{Case 1}, but we have to retake values of $v^N_0$ near interface points on $\Gamma (0)$ in order that the derivative growth $| \nabla^N v^N_0 (x) | \le C_0 K$ holds. This can be done through the following procedure. First we sample values of $v^N_0$ on points which have distance larger than $1/2K$ from two points in $\Gamma (0)$ by the same manner as in \textbf{Case 1} and then we linearly interpolate values of $v^N_0$ on other remaining points. Then we
can easily see that this construction provides us an example of approximating functions $u^N_0$ and $v^N_0$ which satisfy the assumptions (A1) and (B2) simultaneously.

\begin{figure}[h]
 \centering
 \includegraphics[width=14cm]{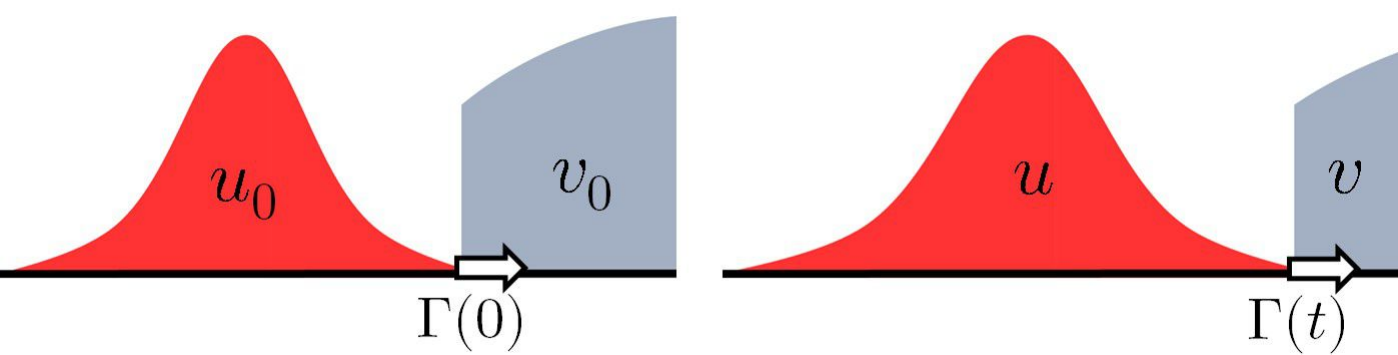}
 \caption{Interface behavior for \textbf{Case 2}. The left figure shows an example of initial functions $u_0$ and $v_0$, while the right one shows a typical situation at some positive time $t $. }
 \label{macrocase2}
\end{figure}

Finally, if reaction rates $c_1$ and $c_2$ are of \textbf{Case 3}, then we get immovable behavior of the limiting interface as follows. 

\begin{theorem}[Immovable interface]
\label{main thm3}
Assume reaction rates $c_1$ and $c_2$ are of \textbf{Case 3}. Assume (A1), (A2), (A3)$_\delta$ and (B3) with some $\delta > 0 $. Then there exists a subsequence $(N_k)$ of $(N)$ and a pair of real-valued functions $ (u, v) $ and a functional $\zeta $ on $Q_T$ such that for every $\varepsilon > 0$ and $\varphi = ( \varphi_1 ,\varphi_2 ) \in C^\infty( Q_T ; \mathbb{R}^2)$ we have  
\[
\lim_{N\to \infty} \mathbb{P}^N_{\mu^N_0} \left( \left| \int_0^T \big( \langle\pi^N_{1,t}, \varphi (t, \cdot )  \rangle - \langle ( u (t,\cdot), v (t,\cdot) ) , \varphi (t, \cdot)  \rangle_{L^2(\mathbb{T}^d; \mathbb{R}^2 )} \big) dt \right| > \varepsilon  \right) =0
\]
and
\begin{equation*}
\begin{aligned}
&u, u^{m/2} \in L^\infty(Q_T) \cap L^2(0,T;H^1(\mathbb{T}^d)), v \in L^\infty (Q_T), \zeta \in H^{-1} (Q_T), \\
&0 \leq u \leq 1, 0 \leq v \leq 1, uv=0 \text{ a.e. in } Q_T, \quad \zeta \geq 0 \text{ in } H^{-1} (Q_T)
\end{aligned}
\end{equation*}
satisfying
\[
\iint_{Q_T} \left\{ -\left( \frac{u^m}{m}-v \right) \varphi_t + \frac{2}{m} u^{\frac{m}{2}} \nabla u^{\frac{m}{2}} \cdot \nabla \varphi \right\} d\theta dt + \frac{4(m-1)}{m^2}  {}_{H^{-1}(Q_T)}\langle \zeta, \varphi \rangle_{H^1_0 (Q_T)}=0 
\]
for all $\varphi \in H^1_0 (Q_T)$.  
Furthermore, assume the same conditions stated in Proposition \ref{immovable}. Then $u, v$ and $\zeta$ satisfy the followings:
\[
\begin{aligned}
&V \equiv 0  &&\text{ on } \Gamma \\
&\partial_t u = \Delta u && \text{ in } (0, T] \times  \Omega^u (0)  \\
&u=0 && \text{ on } (0, T] \times \Gamma(0)  \\
&v=v_0 ,\quad  \zeta = | \nabla u^{m/2} |^2 && \text{ in } Q_T  
\end{aligned}
\]
\end{theorem}

Figure \ref{macrocase3} explains a non-trivial example of interface evolution corresponding to \textbf{Case 3}. Also for this case, there might exist some jump points for $v_0$ at the interface so that we conduct the same procedure as in \textbf{Case 2} to find $u^N_0$ and $v^N_0$ which fulfill the requirements (A1) and (B3).

\begin{figure}[h]
 \centering
 \includegraphics[width=14cm]{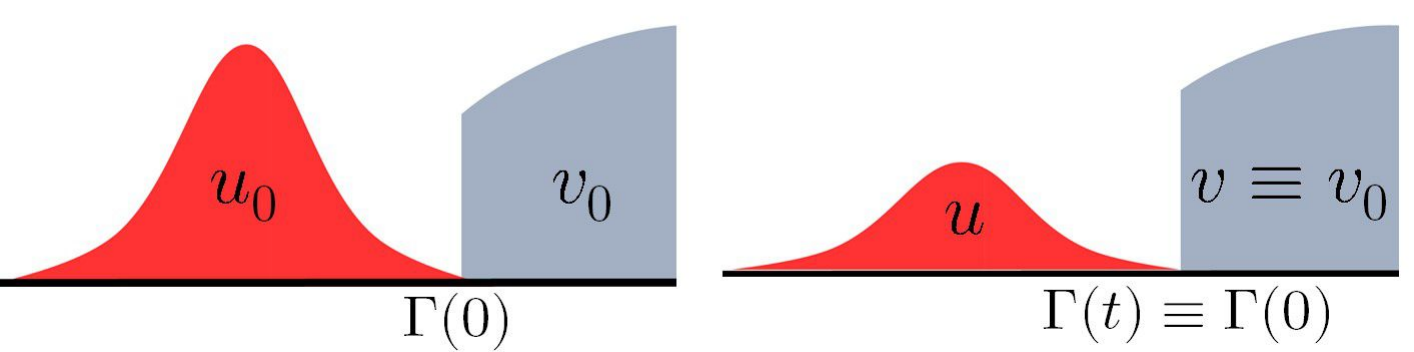}
 \caption{Interface behavior for \textbf{Case 3}. The left figure shows an example of initial functions $u_0$ and $v_0$, while the right one shows a typical situation at some positive time $t $. }
 \label{macrocase3}
\end{figure}

In \textbf{Case 1}, starting from initial densities separated in two particle-phases, though, the initial interface vanishes instantaneously and type-$1$ particles occupy the whole space (see Section \ref{sec:case1}). This is intuitively caused by weak killing effect of type-$2$ particles (recall that time evolution of type-$2$ particles is composed only of annihilation since they do not diffuse), which makes them die out in an instant. On the other hand, in \textbf{Case 2} and \textbf{Case 3}, interfaces are created and they move according to the two-phase Stefan problem in \textbf{Case 2} (see Section \ref{sec:case2}) while they does not move at all in \textbf{Case 3} (see Section \ref{sec:case3}). In \textbf{Case 2}, roughly speaking, multiplication of some monomial of the density of type-$2$ particles to the second equation of (\ref{dHDL eq}) enables us to obtain another reaction-diffusion system whose reaction terms coincide up to a positive constant (but we have to replace a locally scattered product into a spatially-homogeneous one by using uniform boundedness of spatial derivatives). Therefore, this case is essentially considered in \cite{DMFPV19} but the diffusion coefficient for type-$2$ particles is zero in our model and consequently we get a one-phase free boundary problem instead. Finally, in \textbf{Case 3}, type-$1$ particles have weak killing effect. Comparing to \textbf{Case 1}, it seems that type-$1$ particles die out. However, since type-$2$ particles has no ability to diffuse, they do not invade the territory of type-$1$ particles so that the initial interface does not move at all.

\section{Strategy of proof}
\label{sec:strategy}
\subsection{Yau's relative entropy method}
As we noted at the beginning of this paper, the proof of our main theorems is based on Yau's relative entropy method which is introduced in \cite{Yau91} combined with a (deterministic) result which ensures a solution of a semi-discretized system converges to a solution of a continuous PDE system . In this section, we explain what we need to prove the main theorems in detail. To see that, we first introduce a discretized version of macroscopic equations which characterize time evolution of density. For every fixed $T > 0$, let  $u^N= u^N (t,x) $ and $v^N = v^N (t,x) $ be a unique solution of the following semi-discretized  reaction-diffusion system
\begin{equation}
\label{dHDL eq}
\begin{cases}
\partial_t u^N (t, x) = \Delta^N u^N (t,x) -K(N) c_{1, x} ( u^N(t), v^N(t)) u^N (t,x) v^N (t,x) \\
\partial_t v^N (t, x) =                             -K(N) c_{2, x}  (u^N(t) ) u^N (t,x) v^N (t,x)  
\end{cases}
\end{equation}
for every $t \in [0, T]$ and $x \in \mathbb{T}^d_N$. Here $\Delta^N$ denotes the discrete Laplacian defined by 
\[
\Delta^N u(x) \coloneqq N^2 \sum_{y \in \mathbb{T}^d_N, |x-y|=1} \left( u(y)-u(x)  \right) 
\]
for every $u=\{ u(x) \}_{x \in \mathbb{T}^d_N}$ and $c_{i, x}$ is defined by $c_{i, x} ( u (t) ,v (t) ) = c_i (\tau_x u (t), \tau_x v (t) )$ for each $i = 1,2$ and every $[0,1]$-valued functions $u=\{ u (x) \}_{x \in \mathbb{T}^d_N}$ and  $v = \{ v (x) \}_{x \in \mathbb{T}^d_N}$, and non-negative functionals $c_1$ and $c_2$ on $[0,1]^{ \mathbb{T}^d_N } \times [0,1]^{ \mathbb{T}^d_N }$ defined in Section \ref{sec:model}. The semi-discretized system (\ref{dHDL eq}) is a system of ODEs whose solutions are contained in the interval $(0,1)$ uniformly in $(t,x ) \in [0,T] \times \mathbb{T}^d_N$ for every $N \in \mathbb{N}$ and thus it has a unique pair of time-global solution. Indeed, in Lemma \ref{dsol est2} we prove $ 0 < u^N (t, x), v^N (t,x) < 1 $ for every $t \in [ 0, T]$ and $x \in \mathbb{T}^d_N$ provided $0 < u^N (0, x), v^N (0, x) < 1$ for all $x \in \mathbb{T}^d_N$.

In Section \ref{sec:PDE}, we give some general estimates derived for the solution of this semi-discretized system (\ref{dHDL eq}). Moreover, let $\{ u^N(t,\theta) \}_{t \in [0,T], \theta \in \mathbb{T}^d}$ and $\{ v^N(t,\theta) \}_{t \in [0,T], \theta \in \mathbb{T}^d}$ be macroscopic functions on $Q_T:= [0,T]\times \mathbb{T}^d$ defined by
\begin{equation}
\begin{aligned}
\label{dsol}
u^N(t,\theta) := \sum_{x \in \mathbb{T}^d_N} u^N(t, x) \prod_{i=1}^d \mathbf{1}_{\left[ \frac{x_i}{N}-\frac{1}{2N}, \frac{x_i}{N}+\frac{1}{2N} \right)}(\theta_i), \\
v^N(t,\theta) := \sum_{x \in \mathbb{T}^d_N} v^N(t, x) \prod_{i=1}^d \mathbf{1}_{\left[ \frac{x_i}{N}-\frac{1}{2N}, \frac{x_i}{N}+\frac{1}{2N} \right)}(\theta_i)
\end{aligned}
\end{equation}
for every $t \in [0,T]$ and $\theta=(\theta_i)_{1 \leq i \leq d} \in \mathbb{T}^d$.

We prove Theorems \ref{main thm1}, \ref{main thm2} and \ref{main thm3} with the relative entropy method introduced in \cite{Yau91} at hand. Let $\mu^N_t$ be the probability distribution of $\eta^N_t=(\eta^N_{1,t}, \eta^N_{2,t})$ on $\mathcal{X}^2_N$ and let $\nu^N_t$ be the Bernoulli measure on $\mathcal{X}^2_N$ with mean $(u^N(t), v^N(t))$ for $u^N (t) =\{ u^N(t,x)  \}_{x \in \mathbb{T}^d_N}$ and $v^N (t) =\{ v^N(t,x)  \}_{x \in \mathbb{T}^d_N}$. By Lemma \ref{dsol est}, for each fixed $N \in \mathbb{N}$, values of $u^N(t,x)$ and $v^N(t,x)$ are contained in the interval $[0,1]$ provided $0 \le u^N_0(x), v^N_0(x) \le 1$ at initial time, which makes our definition of $\nu^N_t$ well-defined. In this section, we show the next result which states that the distribution of the microscopic dynamics $\{\eta^N_t\}_{t \geq 0}$ is closely described by the semi-discretized system (\ref{dHDL eq}) asymptotically as $N$ tends to infinity. This plays an essential role to prove our main theorems.

\begin{theorem}
\label{prob part}
Assume (A1), (A2) and (A3)$_\delta$ with $\delta = \delta ( T ) > 0$ sufficiently small. Then for any $t \in [0,T]$ we have $H(\mu^N_t | \nu^N_t) = o (N^d) $ as $N \to \infty$.
\end{theorem}

\subsection{Proof of Theorem \ref{main thm1}, \ref{main thm2} and \ref{main thm3}}
Once the main ingredient of probabilistic part Theorem \ref{prob part} is proved, we can deduce the main theorems as follows. Let $u^N $ and $v^N $ be functions on $Q_T$ defined by (\ref{dsol}). For any $\varepsilon > 0$ and any smooth test function $\psi \in C^\infty ( \mathbb{T}^d )$, let us define
\[
\begin{aligned}
\mathcal{A}_1 = \mathcal{A}_1( \psi, \varepsilon  ) \coloneqq \{ \eta \in \mathcal{X}^2_N ; \left| \langle \pi^N_{1, t} , \psi \rangle - \langle u^N (t,\cdot) , \psi \rangle_{L^2(\mathbb{T}^d) }  \right| > \varepsilon \},  \\
\mathcal{A}_2 = \mathcal{A}_1( \psi, \varepsilon  ) \coloneqq \{ \eta \in \mathcal{X}^2_N ; \left| \langle \pi^N_{2, t} , \psi \rangle - \langle v^N (t,\cdot) , \psi \rangle_{L^2(\mathbb{T}^d) }  \right| > \varepsilon \}. 
\end{aligned}
\]
Then, as a corollary of the entropy inequality, we get
\[
\mu^N_t (\mathcal{A}_i) \leq \frac{\log 2 + H ( \mu^N_t | \nu^N_t) }{\log(1+ 1/ \nu^N_t (\mathcal{A}_i))}
\]
for each $i = 1, 2 $. Moreover, for the probability of $\mathcal{A}_i$ under the product Bernoulli measure $\nu^N_t$ in the denominator of the above inequality can be estimated as follows. 

\begin{lemma}
\label{nu est}
For any $\psi \in C^\infty (\mathbb{T}^d )$ and $\varepsilon > 0$, there exists a positive constant $C= C (\varepsilon, \| \psi \|_{  L^\infty (\mathbb{T}^d ) })$ such that
\[
\nu^N_t (\mathcal{A}_i (\psi,  \varepsilon ) ) \leq e^{-C N^d} 
.\]
In particular, the above estimate holds uniformly in $\{ \psi ;  \| \psi \|_{ L^\infty (\mathbb{T}^d ) } < M \} $ for every $M >0$. 
\end{lemma}
The proof of Lemma \ref{nu est} can be done in the same manner as \cite{DMFPV19} so that we omit the proof here. Recalling $H (\mu^N_t | \nu^N_t ) = o(N^d)$ by Theorem \ref{prob part}, we have 
\[
\lim_{N \to \infty} \mu^N_t (\mathcal{A}_i ( \varphi_i (t, \cdot), \varepsilon  ) ) = 0
\]
for each $i=1,2$, $t \in [0,T]$, $\varepsilon > 0$ and $\varphi = (\varphi_1, \varphi_2 )  \in C^\infty ( Q_T ; \mathbb{R}^2 )$. Therefore, once the proof of Theorem \ref{prob part} is completed, the detailed proof of Theorem \ref{prob part} is given in Section \ref{sec:prob}, then we can give the proof of the probabilistic part of our main theorems as follows. First, the probability appearing in the main theorems (Theorem \ref{main thm1}, \ref{main thm2} and \ref{main thm3}) can be estimated by using Markov's inequality the triangle inequality as 
\begin{equation}
\label{main thm proof}
\begin{aligned}
& \frac{1}{\varepsilon } \int_0^T E_{ \mu^N_t } \big[  | \langle \pi^N_t , \varphi (t, \cdot )  \rangle - \langle ( u^N  (t, \cdot ), v^N  (t, \cdot ) ) , \varphi \rangle_{ L^2 (\mathbb{T}^d ) } | \big] dt \\
& + \frac{1}{ \varepsilon } \bigg| \int_0^T  \langle (u^N (t, \cdot ) -u (t, \cdot ), v^N (t, \cdot ) - v (t, \cdot )) , \varphi  (t, \cdot ) \rangle_{ L^2 (\mathbb{T}^d) } dt \bigg| .   
\end{aligned}
\end{equation}
We can see that these two terms converges to zero as $N$ tends to infinity by combining with results for limiting behavior of semi-discretized reaction-diffusion system (\ref{dHDL eq}) given in Section \ref{sec:case1}, \ref{sec:case2} and \ref{sec:case3} for \textbf{Case 1}, \textbf{Case 2} and \textbf{Case 3}, respectively to complete the proof of Theorems \ref{main thm1}, \ref{main thm2} and \ref{main thm3}. Indeed, we have at least $u^N \rightharpoonup u $ and $v^N \rightharpoonup v$ weakly in $L^2 (Q_T)$ for all cases (see Theorems \ref{case1 thm}, \ref{case2 thm} and \ref{case3 thm}) but one should take a subsequence for \textbf{Case 3}. In particular, the second term in \eqref{main thm proof} vanishes as $N$ tends to infinity. On the other hand, the integrand in the first term can be bounded above by 
\[
\begin{aligned}
E_{ \mu^N_t } \bigg[  \big| \langle \pi^N_t , \varphi \rangle - \langle ( u^N (t, \cdot ), v^N (t, \cdot ) ) , \varphi (t, \cdot) \rangle_{ L^2 (\mathbb{T}^d ) } \big| , \bigcap_{ i = 1, 2 } \mathcal{A}_i (\varphi_i ( t, \cdot ), \tilde{ \varepsilon } ) \bigg] + \tilde{\varepsilon} .
\end{aligned}
\]
However, the first term in the above display converges to zero as $N$ tends to infinity since $\lim_{N \to \infty } \mu^N_t (\mathcal{A}_i ) = 0 $ for each $i = 1, 2$ as we proved at the beginning of this subsection and the quantity inside the expectation is bounded above by a positive constant. Therefore, by taking $\tilde{ \varepsilon } > 0$ small enough to complete the proof.


\section{Proof of Theorem \ref{prob part}}
\label{sec:prob}
\subsection{The relative entropy method}
In this section, we prove Theorem \ref{prob part}. We first define a Dirichlet energy corresponding to the Kawasaki dynamics with respect to the reference measure $\nu $ (namely $\nu$ is a product Bernoulli measure on $\mathcal{X}^2_N$ with full support) as follows: for any $f: \mathcal{X}^2_N \to \mathbb{R}$, we define  
\[
\mathcal{D} \left( f ; \nu \right) \coloneqq \frac{1}{4} \sum_{x \in \mathbb{T}^d_N} \sum_{j =1 }^d \int_{\mathcal{X}^2_N } 
\left[ f (\eta_1^{x, x+e_j} , \eta_2 ) - f ( \eta_1, \eta_2 ) \right]^2 d\nu ( \eta_1, \eta_2 ) . 
\]
If the reference measure $\nu$ is a product Bernoulli measure with constant weight, then the above energy becomes the Dirichlet form corresponding to our Kawasaki dynamics. Recall here that we only have the Kawasaki dynamics for type-$1$ particles. We then have the following estimate on entropy production (time derivative of relative entropy). 

\begin{proposition}[Yau's inequality, \cite{JM18}]
\label{Yau}
For any probability measures $\{\nu_t \}_{t \geq 0}$ and $m$ on $\mathcal{X}^2_N$ which are differentiable in $t$ and full-supported on $\mathcal{X}^2_N$, we have 
\begin{equation}
\label{Yau ineq}
\frac{d}{dt} H(\mu^N_t | \nu_t) \leq -2 N^2 \mathcal{D} \bigg( \sqrt{ \frac{ d \mu^N_t }{ d \nu_t } }; \nu_t \bigg)
+ \int_{\mathcal{X}^2_N} (L^{*, \nu_t}_N \mathbf{1}- \partial_t \log{\psi_t}) d\mu^N_t 
\end{equation}
\noindent
where $L^{*, \nu_t}_N$ is the adjoint operator of $L_N$ on $L^2 (\nu_t)$ and $\psi_t := d\nu_t/ dm$.
\end{proposition}

We define scaled variables $\omega_{i,x}(t)$ by
\[
\omega_{1, x} \equiv \omega_{1,x}^N (t) \coloneqq \frac{\overline{\eta}_1 (x) }{\chi(u^N(t,x))}, \quad 
\omega_{2, x} \equiv \omega_{2,x}^N (t) \coloneqq \frac{\overline{\eta}_2 (x) }{\chi(v^N(t,x))}, 
\]
with $\overline{\eta}_1 (x) \coloneqq \eta^N_{1,t}(x) - u^N(t,x)$ and $\overline{\eta}_2 (x) \coloneqq \eta^N_{2,t} (x) - v^N(t,x)$. Moreover, $\chi(\rho)= \rho(1-\rho)$ is the incompressibility for $\rho \in [0,1]$. We show in Section 4 that $0 < u^N (t,x), v^N (t,x) < 1$ holds for every $t \in [0,T] $ and $x \in \mathbb{T}^d_N $ if $0 < u^N (0,x), v^N (0, x) < 1$ holds for every $x \in \mathbb{T}^d_N$ and thus the denominator of $\omega_{i,x}$ is always positive and it becomes well-defined for each $i= 1, 2$. In the sequel, we sometimes omit dependence on $t$ or $N$ for notational simplicity only for the case where dependence on those parameters is not important or it is obvious from context.

\begin{lemma}
\label{Yau integrand}
We have 
\[
L^{*,\nu^N_t}_N \mathbf{1}- \partial_t \log{\psi_t}= V_K(t) + V_G(t)
\]
with
\[
\begin{aligned}
 V_K(t) & = -\frac{N^2}{2} \sum_{x,y \in \mathbb{T}^d_N, |x-y|=1} {\left( u^N (y) - u^N (x) \right)}^2 \omega_{1, x} \omega_{1, y}, \\    
 V_G(t) & =- K \sum_{x \in \mathbb{T}^d_N} \big[ c_{1,x} (\eta_1, \eta_2)  \eta_2 (x) - c_{1,x} (u^N, v^N) v^N (x)  \big]  u^N (x) \omega_{1,x} \\ 
& \quad \,  - K \sum_{ x \in \mathbb{T}^d_N }  \big[ c_{2,x} (\eta_1 ) \eta_1(x) - c_{2,x} (u^N) u^N (x) \big] v^N (x) \omega_{2,x} 
\end{aligned}
\]
and these do not depend on particular choice of the reference measure $m$ on $\mathcal{X}^2_N$. In particular, when the Glauber part has the form of \textbf{Case 1}, 
\[
\begin{aligned}
V_G (t) = 
& - K \sum_{x \in \mathbb{T}^d_N } \left[ \eta_2 (x) \eta_1 (x_1) \cdots \eta_1 (x_{m-1}) - v^N (x) u^N (x_1) \cdots u^N ( x_{m-1}) \right] u^N (x) \omega_{1,x}  \\ 
& - K \sum_{x \in \mathbb{T}^d_N } \chi ( u^N (x) )  v^N (x) \omega_{1,x} \omega_{2, x } , 
\end{aligned}
\] 
for \textbf{Case 2}, 
\[
\begin{aligned}
V_G (t) = 
& - K \sum_{x \in \mathbb{T}^d_N } \left[ \eta_2 (x) \eta_2 (x_1) \cdots \eta_2 (x_{m-1}) - v^N (x) v^N (x_1) \cdots v^N (x_{m-1}) \right] u^N (x) \omega_{1,x} \\
& - K \sum_{x \in \mathbb{T}^d_N } \chi ( u^N (x) ) v^N (x)  \omega_{1,x} \omega_{2, x } , 
\end{aligned}
\]
and for \textbf{Case 3}, 
\[
\begin{aligned}
V_G (t) = 
& - K \sum_{x \in \mathbb{T}^d_N } u^N (x) \chi ( v^N (x) )  \omega_{1,x} \omega_{2, x }  \\
& - K \sum_{x \in \mathbb{T}^d_N } \left[ \eta_1 (x) \eta_1 (x_1) \cdots \eta_1 (x_{m-1}) - u^N (x) u^N (x_1) \cdots u^N (t, x_{m-1}) \right] v^N (x) \omega_{2, x } , 
\end{aligned}
\] 
respectively. Here we have set $x_i \coloneqq x + z_i$ for $i = 1,..., m-1$ for simplicity. 
\end{lemma}

\begin{remark}
\label{c2 reason}
Linear terms in $\omega$ cancel by the semi-discretized system (\ref{dHDL eq}) and hence the above $V_K$ and $V_G$ are reminder terms. Since we assumed the reaction rate $c_2$ did not depend on configuration of type-$2$ particles, any higher order correlation between $\omega_2$'s does not appear in $V_G$. Such terms cause appearance of a Dirichlet energy with respect to the Kawasaki dynamics for type-$2$ particles with positive sign and it can not be absorbed by the first term in (\ref{Yau ineq}). That is why we have assumed $c_2$ to be a function of configuration of only type-$1$ particles.   
\end{remark}

\begin{proof}
First we calculate for the Glauber part. For any $f : \mathcal{X}^2_N \to \mathbb{R}$, we have that $E_{\nu^N_t} [f L^{*, \nu^N_t}_{1,G} \mathbf{1}]= E_{ \nu^N_t } [L_{1,G} f ]$ is equal to
\begin{equation}
\label{L1G}
\sum_{\eta_1,\eta_2} \sum_{x \in \mathbb{T}^d_N} c_{1,x}(\eta_1, \eta_2) \eta_1 (x) \eta_2 (x) 
 \left[ f(\eta^{x}_1,\eta_2) - f(\eta_1, \eta_2) \right] \nu^N_t (\eta_1, \eta_2).
\end{equation}
Recalling the form of reaction rates, we observe that for any configuration $\eta_1 $ 
\[
\eta^x_1 (x) \nu^N_t (\eta^x_1, \eta_2) = \frac{u^N (x) }{ 1 - u^N (x) } \big( 1- \eta_1 (x) \big) \nu^N_t (\eta_1, \eta_2).
\]
Since $c_{1,x}(\eta_1, \eta_2)$ does not depend on $\eta_1(x)$, it is invariant under change of variables $\eta^x_1 \mapsto \eta_1$. Therefore, (\ref{L1G}) further equals to 
\[
\sum_{\eta_1, \eta_2} \sum_{x\in \mathbb{T}^d_N} c_{1,x}(\eta_1, \eta_2) f(\eta_1, \eta_2) \eta_2 (x) 
\left[ \frac{u^N (x)}{ 1 - u^N (x)} ( 1 - \eta_1(x) ) - \eta_1(x) \right] \nu^N_t (\eta_1, \eta_2) 
.\]
Since $f: \mathcal{X}^2_N \to \mathbb{R}$ is arbitrary, we thus obtain
\[
\begin{aligned}
L^{*,\nu^N_t}_{1, G} \mathbf{1} 
&=  \sum_{x\in \mathbb{T}^d_N} c_{1,x}(\eta_1, \eta_2) \eta_2 (x) 
\left[ \frac{ u^N (x) }{ 1 - u^N (x) }( 1 - \overline{\eta}_1 (x) - u^N (x) ) -(\overline{\eta}_1 (x) + u^N (x)) \right]  \\
&= - \sum_{x \in \mathbb{T}^d_N} c_{1,x} (\eta_1, \eta_2) \eta_2 (x) \frac{ \overline{\eta}_1 (x) }{ 1 - u^N (x) }   \\
&= - \sum_{x \in \mathbb{T}^d_N} \left[ c_{1,x} (\eta_1, \eta_2)  \eta_2 (x) - c_{1,x} (u^N, v^N) v^N (x)  \right]  u^N (x) \omega_{1,x} \\
& \quad - \sum_{x \in \mathbb{T}^d_N} c_{1,x} (u^N, v^N) u^N (x) v^N (x) \omega_{1,x}. 
\end{aligned}
\]
Note here that the second term is linear in $\omega_1$ and the first term has higher order which is equal to the first term of $V_G (t)$. Also, $L^{*, \nu^N_t}_{2, G} \mathbf{ 1 }$ is calculated in the same manner as follows: Recalling $c_2$ depends only on the configuration of type-$1$ particles, 
\[
\begin{aligned}
L^{*,\nu^N_t}_{2, G} \mathbf{1} 
&= - \sum_{x \in \mathbb{T}^d_N} \left[ c_{2,x} (\eta_1)  \eta_1 (x) - c_{2,x} (u^N ) u^N (x)  \right]  v^N (x) \omega_{2,x} \\
& \quad - \sum_{x \in \mathbb{T}^d_N} c_{2, x} (u^N ) u^N (x) v^N (x) \omega_{2,x} 
\end{aligned}
\]
and the higher order term matches the second term of $V_G (t ) $. For the Kawasaki part, by a similar calculation given in \cite{FT18}, we can easily obtain
\[
N^2 L^{*,\nu^N_t}_K \mathbf{1} =-\frac{N^2}{2} \sum_{x,y\in \mathbb{T}^d_N, |x-y|=1} \left[ u^N(y) - u^N(x) \right]^2 \omega_{1,x} \omega_{1,y} +\sum_{x\in \mathbb{T}^d_N} \Delta^N u^N(x)\omega_{1,x} .
\]
Finally, a simple computation similar to \cite{FT18} yields 
\[
\partial_t \log{\psi_t} (\eta) 
= \sum_{x \in \mathbb{T}^d_N} \partial_t u^N(x) \omega_{1,x}
+ \sum_{x \in \mathbb{T}^d_N} \partial_t v^N(x) \omega_{2,x}
.\]
Therefore, we could represent the integrand appearing in Yau's inequality as the polynomial expansion of $\omega_i$ but linear terms in $\omega_i$ $(i=1,2)$ cancel by our semi-discretized reaction-diffusion system (\ref{dHDL eq}) so that we end the proof. 
\end{proof}

\begin{theorem}
\label{est V}
We assume the same conditions as Theorem \ref{prob part}. Let $d \geq 2$. Then, for any $\alpha>0$ and $0 <\kappa < 1$, there exists a positive constant $ C $ depending only on $\alpha$ and $\kappa$ such that
\begin{equation}
\label{est V_G}
E_{\mu^N_t} \left[ V_G(t) \right] 
\leq \alpha N^2 \mathcal{D} (\sqrt{f}; \nu^N_t) + C K H( \mu^N_t |\nu^N_t ) + N^{ d-1+\kappa }  
\end{equation}
and also
\begin{equation}
\label{est V_K}
E_{\mu^N_t} \left[ V_K(t) \right]
\leq \alpha N^2 \mathcal{D} (\sqrt{f}; \nu^N_t) + C K^2 H( \mu^N_t |\nu^N_t ) + N^{ d-1+\kappa }  
.\end{equation}
When $d=1$, the last terms $N^{d-1+\kappa}$ in both estimates are replaced by $N^{ 1/2 + \kappa }$. 
\end{theorem} 

The proof of this theorem is postponed in the nest subsection and we first give the proof of Theorem \ref{prob part}.

\begin{proof}[Proof of Theorem \ref{prob part}]
We now combine Theorem \ref{est V} and Yau's inequality (Proposition \ref{Yau}) to end the proof of Theorem \ref{prob part}. We choose $\alpha > 0$ so that the Dirichlet form with positive coefficient can be absorbed into the first term of (\ref{Yau ineq}), which enables us to estimate
\[
\frac{d}{dt} H( \mu^N_t | \nu^N_t ) \leq C K^2  H ( \mu^N_t | \nu^N_t ) + O (N^{d-\delta_1})
\]
with some $0 < \delta_1 < 1$. Therefore, by Gronwall's inequality, we have
\[
H( \mu^N_t | \nu^N_t ) \leq \big(  H( \mu^N_0 | \nu^N_0 ) + t O (N^{d-\delta_1}) \big)  e^{ C K^2 t }
.\]
Now by the assumption (A2) and (A3)$_\delta$ with $\delta> 0$ small enough (in order that $\delta^2 < \delta / C T $ holds), we end the proof of Theorem \ref{prob part}. 
\end{proof}

\subsection{Proof of (\ref{est V_G})}
First we consider terms which appear in $V_G (t) $. Since $c_1 $ and $c_2 $ are polynomial of configuration, the residual term $V_G $ is a linear combination of the form 
\[
K \sum_{ x \in \mathbb{T}^d_N } \overline{\eta}_{1, x + \Lambda_1}  \overline{\eta}_{2,x + \Lambda_2 }
\]
where $\overline{\eta}_{i,  x + \Lambda_i} = \prod_{ y \in \Lambda_i } \overline{\eta}_{ i ,x + y }$ for $i =1,2$ and $\Lambda_1 , \Lambda_2 \Subset \mathbb{Z}^d $ with $|\Lambda_1| \ge 1$ and $|\Lambda_1 | + | \Lambda_2| \ge 2$. For this term, we take ``the utmost right site''  in $\Lambda_1$ and change variables in summation in order that the picked site is again denoted by $x$. In other words, we consider the quantity with the following form:  
\begin{equation}
\label{V}
V = K \sum_{x \in \mathbb{T}^d_N} g_x ( \eta_1, \eta_2  ) \omega_{1, x } 
\end{equation}
for some functional $g (\eta )=  g ( \eta_1 , \eta_2 ) $ such that $g_x ( \eta ) =\tau_x g (\eta) $ for every $x \in \mathbb{T}^d_N$ and $g_{x- y} (\eta )$ is invariant under the transformation $\eta_1 \mapsto \eta_1^{x, x+e_j}$ for every $y \in \Lambda_{2 \ell } = [0,2 \ell-1] ^d  \cap \mathbb{Z}^d$ and $j = 1,...,d$. Moreover, we suppose the function $g(\eta)$ has the bound $\| g \|_{L^\infty} \le C e^{ C_1 K }$ for some positive constant $C_1$. It is noted here that such function $g$ is bounded uniformly in $N$ when it is calculated for $V_G$ since any term in $V_G$ has neither $u^N$ nor $v^N$ in its denominator. However, when a multi-point correlation which comes from $V_K$ is considered, it has at least $u^N$ or $v^N$ in its denominator. In particular, according to Lemma \ref{dsol est2}, the above bound has to be assumed. Hence we impose the above bound in advance in order to make all cases to be proved at once.

The first step to prove Theorem \ref{est V} is to replace $V$ by its local average $V^\ell$ defined by 
\[
V^\ell = K \sum_{x \in \mathbb{T}^d_N}  \overleftarrow{  g(\eta ) }_{x,\ell} \overrightarrow{(\omega_1)}_{ x , \ell } 
\]
where 
\[
\overleftarrow{G}_{x,\ell} \coloneqq \frac{1}{ | \Lambda_\ell | } \sum_{y \in \Lambda_\ell} G_{x-y}, \quad
\overrightarrow{G}_{x,\ell} \coloneqq \frac{1}{ | \Lambda_\ell | } \sum_{y \in \Lambda_\ell} G_{x+y} \quad
\]
for $G=\{ G_x \}_{x \in \mathbb{T}^d_N}$ and $\Lambda_\ell = [0,\ell-1] ^d  \cap \mathbb{Z}^d$. Then we can estimate the cost to replace $V$ by its local average $V^\ell$ as follows.

\begin{lemma}
\label{replacement}
We assume the same conditions as Theorem \ref{prob part} and choose $\ell= N^{1/d-\kappa/d}$ when $d \geq 2$ and $\ell= N^{1/2-\kappa}$ when $d=1$ with $\kappa>0$ sufficiently small. Then there exists a positive constant $C$ depending only on $\alpha$ and $\kappa$ such that
\[
E_{\mu^N_t} \left[ V- V^\ell \right] \leq \alpha N^2 \mathcal{D} (\sqrt{f} ; \nu^N_t) + C \left(H (\mu^N_t | \nu^N_t ) + N^{d-1+\kappa} \right)
\] 
for every $\alpha >0$ when $d \geq 2$ and the last term $N^{d-1+\kappa}$ is replaced by $N^{1/2+\kappa}$ when $d=1$. 
\end{lemma}

To prove this lemma, we use the following key estimate between two probability measures which is called \textit{flow lemma} introduced in \cite{JM18}. To state the flow lemma, we introduce the notion of a flow between two probability measures on a graph. 

\begin{definition}
Let $G = (V, E)$ be a finite graph where $V$ is a set of all vertices and $E$ is the set of all edges. For two probability measures $p,q$ on $V$, we call $\Phi= \{ \Phi(x,y) \}_{ \{x,y\} \in E }$ a flow on $G$ connecting $p$ and $q$ if it satisfies:
\begin{itemize}
\item $\Phi(y,x)=- \Phi(x,y)$ for all $\{ x,y \} \in E$,
\item $\sum_{z \in V } \Phi(x,z)= p(x) -q(x)$ holds for all $x \in V $.
\end{itemize}
\end{definition}

In the sequel, we regard any finite subset in $\mathbb{Z}^d$ as a graph where the set of all bonds means the set of all pair of two points in that set such that the Euclidean distance between them is $1$.

\begin{proposition}[Flow lemma]
\label{flow lem}
Let $\delta_0$ be the Dirac measure on $\mathbb{Z}^d$ with mass $1$ on $0 \in \mathbb{Z}^d$ and let $ p_\ell $ be the uniform probability measure on $ \mathbb{Z}^d  $ with mass on $\Lambda_\ell $ defined by $p_\ell (x)= |\Lambda_\ell|^{-1} \mathbf{1}_{\Lambda_\ell}(x) $. Moreover, let $q_\ell $ be the probability measure on $\mathbb{Z}^d$ defined by $q_\ell (x) = p_\ell * p_\ell (x)  \coloneqq \sum_{ y \in \mathbb{Z}^d } p_\ell (y)  p_\ell (x-y) $. Then there exists a flow $\Phi^\ell$ on $\Lambda_{2\ell}$ connecting $\delta_0$ and $q_\ell$ such that $\Phi^\ell (x,y) = 0$ for any $x \in \Lambda_{2\ell}^c$ and $y \in \mathbb{Z}^d$, and that
\[
\sum_{x \in \Lambda_{2\ell}} \sum_{j=1}^d \Phi^\ell (x,x+e_j)^2  \le C_d g_d (\ell)
\]
where $e_j$ is a unit vector to $j$-th positive direction and $g_d(\ell)$ is given by
\[
g_d (\ell)=
\begin{cases}
\begin{aligned}
&\ell         &&\text{ if } d=1, \\
&\log{\ell} &&\text{ if } d=2, \\
&1 		  &&\text{ if } d\geq 3.
\end{aligned}
\end{cases}
\]
\end{proposition}

In the sequel, we prove Lemma \ref{est V} by using the flow lemma. To see that, one can notice for any $G=\{ G_x \}_{x \in \mathbb{T}^d_N}$ we have
\[
G* p_\ell (x) = \sum_{y \in \mathbb{T}^d_N} G_{x-y} p_\ell(y) = \frac{1}{ | \Lambda_\ell | } \sum_{y \in \Lambda_\ell} G_{x-y} = \overleftarrow{ G }_{x,\ell}
\]
and similarly $G * \hat{p}_\ell (x)= \overrightarrow{ G }_{x,\ell}$ with $\hat{p}_\ell(y) \coloneqq p_\ell (-y)$. Therefore, using the above identity and by definition of convolution, the local average $V^\ell$ can be rewritten as 
\[
\begin{aligned}
V^\ell 
& = K \sum_{x \in \mathbb{T}^d_N }  
\bigg( \sum_{ y \in \mathbb{T}^d_N } g_y (\eta) p_\ell (x-y ) \bigg) 
\bigg( \sum_{ z \in \mathbb{T}^d_N } \omega_{1,z} p_\ell (z-x)  \bigg) \\
& = K \sum_{x, y, z } g_y (\eta) \omega_{1, z} p_\ell (x) p_\ell (z-y-x)  \\
& = K \sum_{y, z } g_y (\eta ) \omega_{ 1, z } \hat{q}_\ell (y-z)  = K \sum_{ x } g_x ( \eta ) (\omega_{1} * \hat{q}_\ell) (x)  
\end{aligned}
\]
where we changed variables $y$ to $x$ in the last line. According to Proposition \ref{flow lem}, we can take a flow connecting $\delta_0$ and $q_\ell$ to calculate the difference between $V$ and its local average $V^\ell$ as  
\[
\begin{aligned}
V- V^\ell 
& = K \sum_{x \in \mathbb{T}^d_N} g_x (\eta) \bigg( \omega_{1,x} - \sum_{y \in \mathbb{T}^d_N } \omega_{1, x + y } q_\ell ( y )  \bigg) \\
& = K \sum_{x, y  } g_x (\eta)  \omega_{1,x + y } \big( \delta_ 0 ( y ) - q_\ell ( y ) \big) \\
& = K \sum_{ x ,y } g_x (\eta) \omega_{ 1, x + y } \sum_{j = 1}^{d}  \big( \Phi^\ell (y, y+e_j) + \Phi^\ell (y, y - e_j) \big)  \\ 
& = K \sum_{j=1}^d \sum_{ x, y } g_x (\eta ) \big( \omega_{1, x  + y }-\omega_{ 1, x+y+e_j } \big) \Phi^\ell (y, y+ e_j)  \\
& = K \sum_{j=1}^d \sum_{x} \bigg( \sum_{ y } g_{x-y} (\eta ) \Phi^\ell (y ,y +e_j ) \bigg) (\omega_{1, x} - \omega_{1, x +e_j} )   
\end{aligned}
\]
where in the penultimate line we used the summation by parts recalling that $\Phi^\ell$ is anti-symmetric by definition and that $\Phi^\ell (x,y) =0$ unless both $x$ and $y$ belong to $\Lambda_{2 \ell }$, and in the last line we again conducted the summation by parts. By this line, we have shown the identity  
\begin{equation}
\label{V-Vl}
V - V^\ell = K \sum_{j=1}^d \sum_{x \in \mathbb{T}^d_N}  h^{\ell, j }_x (\omega_{1,x} - \omega_{1, x+e_j} )
\end{equation}
with
\[
h^{\ell, j}_{x} \equiv h^{\ell,j}_x (\eta_1, \eta_2) = \sum_{y \in \Lambda_{2\ell } } g_{x-y} (\eta_1, \eta_2) \Phi^\ell(y,y+e_j)
.\]
Recalling that we took ``the utmost right site'' $x$ in the definition of $g $ so that $g_{x-y} = g_{x-y} (\eta_1,\eta_2 ) $ is invariant under transformation $\eta_1 \mapsto \eta^{ x, x + e_j } $ for any $y \in \Lambda_{ 2 \ell } $ and $j = 1,...,d$, and so $h^{\ell, j}_x $ also becomes to be invariant under that transformation. Moreover, since $g_x$ and $g_x g_y $ with $x \neq y$ has average zero under $\nu^N_t$, recalling that $g$ is bounded uniformly in $N$, there exists a positive constant $C$ which is independent of $N$ such that 
\[
E_{\nu^N_t } [ h^{\ell, j}_x ] = 0, \, \text{Var}_{\nu^N_t } [ h^{\ell, j }_x ] \le C g_d (\ell) e^{ 2 C_1 K }
\]     
by the flow lemma (Proposition \ref{flow lem}) and the lower bound of $u^N$ according to Lemma \ref{dsol est2}.

We have the following integration by parts formula and an estimate for the cost to replace $V$ by its local average $V^\ell$. These are already proved in \cite{DMFPV19} so that we omit the proof here.

\begin{lemma}[Integration by parts]
\label{IBP}
Let $\nu$ be the Bernoulli measure on $\mathcal{X}^2_N$ with mean $(u,v)$ with $u = \{ u(x) \}_{x \in \mathbb{T}^d_N} $, $v = \{ v(x) \}_{x \in \mathbb{T}^d_N}$ satisfying $0 < u (x), v(x) <1$ and assume there exist some $c_1 >0$ and $0< c_2 <1$ such that $ u(x), u(y) \in [ e^{-c_1 K} , c_2 ] $ for any $x,y \in \mathbb{T}^d_N$ with $|x-y|=1$. Then, for $h= h(\eta)$ satisfying $h(\eta_1^{x,x+e_j}, \eta_2) = h(\eta_1, \eta_2) (x \in \mathbb{T}^d_N, j=1,...,d)$ and for any probability density $f$ with respect to $\nu$, we have 
\[
\int_{\mathcal{X}^2_N} h (\eta) (\eta_{1,y}- \eta_{1,x}) f(\eta) d\nu(\eta)= 
\int_{\mathcal{X}^2_N} h(\eta) \eta_{1,x} \left[ f(\eta_1^{x,y}, \eta_2 ) - f(\eta_1, \eta_2 ) \right]  d\nu (\eta) + R_1 
\] 
for any $x,y\in \mathbb{T}^d_N$ with $|x-y|=1$ and the error term $R_1 $ is bounded as
\[
|R_1| \leq C e^{2C_1 K} | u(x)-u(y) | \int_{\mathcal{X}^2_N} | h(\eta) | f(\eta) d\nu (\eta) 
\]
with some positive constant $C>0$. 
\end{lemma}

\begin{lemma}
\label{IBP2}
Under the same assumptions stated in Lemma \ref{IBP}, we have 
\begin{equation}
\label{summand V-Vl}
\int_{\mathcal{X}^2_N} h^{\ell,j}_x (\omega_{1,z+e_j} - \omega_{1,z}) f d\nu = 
\int_{\mathcal{X}^2_N} h^{\ell,j}_x \frac{\eta_{1,z }}{\chi(u( z ))} \left[ f(\eta_1^{z , z + e_j} , \eta_2) - f(\eta_1, \eta_2)  \right] d\nu + R_2
\end{equation}
for every $x, z \in \mathbb{T}^d_N$ and the error term $R_2$ is bounded as 
\begin{equation}
\label{R bound}
|R_2| \leq  C e^{3C_1 K} |u( z )- u( z + e_j)| \int_{\mathcal{X}^2_N} |h^{\ell,j}_x(\eta)| f d\nu
\end{equation}
with some positive constant $C>0$. 
\end{lemma}

Applying these lemmas, we next bound the summand in (\ref{V-Vl}). Here we write the Dirichlet energy corresponding to the Kawasaki dynamics as a sum of its pieces
\[
\mathcal{D}_{x,x+e_j}( f ; \nu)  \coloneqq \frac{1}{4} \int_{\mathcal{X}^2_N} {\left[ f(\eta^{x,x+e_j}_1, \eta_2) - f(\eta_1, \eta_2)  \right]}^2 d\nu(\eta)
\]
so that $\mathcal{D}( f ; \nu) = \sum_{x \in \mathbb{T}^d_N} \sum_{j=1,...,d} \mathcal{D}_{x,x+e_j}( f; \nu)$. We recall here the definition of $\mu^N_t$ and $\nu^N_t$ given at the beginning of subsection 3.1 and hereafter we define $f^N_t \coloneqq d\mu^N_t / d\nu^N_t$ so that we have $\mu^N_t = f^N_t \nu^N_t $.

\begin{lemma}
\label{summand est}
Assume the assumption (A1). Then there exists a positive constant $C$ such that for every $\beta>0$ and $x, z \in \mathbb{T}^d_N$ we have
\[
\int_{\mathcal{X}^2_N} h^{\ell,j}_x (\omega_{1,z } - \omega_{1,z + e_j}) d\mu^N_t 
\leq \beta \mathcal{D}_{z , z + e_j }(\sqrt{ f^N_t }; \nu^N_t ) + \frac{C}{\beta}  e^{ 3 C_1 K} \int_{\mathcal{X}^2_N} (h^{\ell,j}_x)^2   d\mu^N_t  + R_{ 1, z, j }
\]
and each error term $R_{ 1,z, j }$ satisfies the bound (\ref{R bound}), that is, 
\[
|R_{ 1,z, j }| \leq  C e^{3 C_1 K} |u^N ( z )- u^N ( z + e_j)| \int_{\mathcal{X}^2_N} |h^{\ell,j}_x(\eta)|  d\mu^N_t . 
\]
\end{lemma}
\begin{proof}
After applying Lemma \ref{IBP2} with $h = h^{\ell, j}_x$, $f =f^N_t$ and $\nu= \nu^N_t$, we decompose $f^N_t (\eta_1^{z, z + e_j }, \eta_2) - f^N_t (\eta_1, \eta_2)$ into product by using $a^2-b^2=(a+b)(a-b)$. Then, by an elementary inequality $ab \leq A a^2/ 2  + b^2/ 2A $ for any $a,b \in \mathbb{R}$ and $A >0$, the first term in the right hand side of (\ref{summand V-Vl}) is bounded above by
\begin{equation}
\label{bound}
\beta \mathcal{D}_{z, z + e_j } (\sqrt{f^N_t }; \nu^N_t ) + \frac{C}{\beta \chi(u^N ( z ) )^2} \int_{\mathcal{X}^2_N} (h^{\ell,j}_x)^2 \left[ f^N_t ( \eta_1^{z , z + e_j },\eta_2) + f^N_t (\eta_1, \eta_2) \right]d\nu^N_t .
\end{equation}
Let $\nu_1$ be a product Bernoulli measure on $\mathcal{X}_N$ with weight $u=\{ u(x) \}_{x \in \mathbb{T}^d_N}$ with $0 < u(x) < 1$ for every $x \in \mathbb{T}^d_N$. Taking the spatial-inhomogeneity of $u$ into account, for every $x, y \in \mathbb{T}^d_N $ such that $|x - y| = 1 $, we get the cost to replace $\nu_1(\eta_1^{x,y})$ to $\nu_1(\eta_1)$ as
\[
\frac{\nu_1(\eta_1^{x,y})}{\nu_1(\eta_1)} =1 + r_{x,y}(\eta_1)
\]
with
\[
r_{x,y}( \eta_1 ) =
   \mathbf{1}_{\{ \eta_1( x )  = 1 , \eta_1 ( x ) = 0 \} } \frac{u(y)- u(x)}{u(x)(1-u(x))}  
+ \mathbf{1}_{\{ \eta_1 (x)  =0, \eta_1 (x) = 1  \} } \frac{u(x)- u(y)}{(1-u(x))u(y)}  
\]
and this error to change variables can be absolutely bounded as
\[
|r_{x,y}(\eta_1)| \leq C_0 e^{ C_1 K} | u(x)- u(y) |
\]
for some positive constant $ C_0 $ by our assumption on $u$. Therefore, by conducting the change of variable $\eta_1^{ z ,z + e_j } \mapsto \eta_1$ and using the bound of the cost $r_{z, z + e_j }$, the integral in (\ref{bound}) divided by $\chi(u^N ( z ))^2$ is bounded above by
\[
\frac{1+C_0 e^{ C_1 K} |u^N ( z )- u^N ( z + e_j )|}{\chi( u^N ( z ) )^2} \int_{\mathcal{X}^2_N} (h^{\ell,j}_x)^2 f^N_t d\nu^N_t 
.\]
Hence, recalling the definition of the incompressibility $\chi(\cdot)$ and using the bound for $u$ to end the proof.
\end{proof}

Now we prove Lemma \ref{replacement} by using the concentration inequality which is used in vast literatures.

\begin{proposition}[Concentration inequality]
\label{concentration}
Let $\{X_i \}_{\{i=1,...,n\}}$ be a sequence of independent random variables such that each $X_i$ takes values in the interval $[a_i, b_i]$ for $a_i, b_i \in \mathbb{R}$ with $a_i < b_i$. Set $\bar{X}_i = X_i - E[X_i]$ and $\kappa = \sum_{i=1}^n (b_i - a_i )^2$. Then, for every $\gamma \in [0, \kappa^{-1}]$, we have 
\[
\log{ E \bigg[  e^{ \gamma \big(\sum_{i=1,...,n} \bar{X}_i \big )^2 }  \bigg]  } \leq 2\gamma \kappa
.\] 
\end{proposition}

\begin{proof}[Proof of Lemma \ref{replacement}]
Recalling the representation of $V-V^\ell$ in (\ref{V-Vl}), what we should estimate is given by
\[
\int_{\mathcal{X}^2_N} (V- V^\ell) d\mu^N_t 
= K \sum_{j=1}^d \sum_{x \in \mathbb{T}^d_N}  \int_{\mathcal{X}^2_N}  h^{\ell,j}_x (\omega_{1,x}- \omega_{1,x+e_j}) d\mu^N_t 
.\]
By Lemma \ref{summand est}, taking $\beta = \alpha N^2 K^{-1}$ with $\alpha >0$, the above quantity is bounded above by
\[
\begin{aligned}
\alpha N^2 \mathcal{D} (\sqrt{f^N_t }; \nu^N_t ) 
+ \frac{CK^2}{\alpha N^2} e^{3 C_1 K} \sum_{j=1}^d \sum_{x \in \mathbb{T}^d_N}  \int_{\mathcal{X}^2_N}  (h^{\ell,j}_x)^2 d\mu^N_t  + K \sum_{j=1}^d \sum_{x \in \mathbb{T}^d_N} R_{1, x , j }.
\end{aligned}
\]
Recall that the residual term $R_{1,x,j}$ has the bound (\ref{R bound}) for each $x \in \mathbb{T}^d_N$ and $j=1,...,d$. Since $| u^N (x)- u^N (x+e_j) | \leq C K N^{-1}  $ by Lemma \ref{grad u}, estimating $| h^{\ell,j}_x | \leq  1 + (h^{\ell,j}_x)^2 $, we have 
\[
K |R_{1,x,j} | \leq \frac{ C K^2 }{ N } e^{3 C_1 K} \int_{\mathcal{X}^d_N} \left( 1+ (h^{\ell,j}_x)^2 \right) d\mu^N_t 
.\]
Therefore, the expectation with respect to $\mu^N$ of $V - V^\ell$ is bounded above by
\[
\alpha N^2 \mathcal{D}(\sqrt{ f^N_t }; \nu^N_t ) 
+ \frac{C_\alpha K^2}{N} e^{3 C_1 K} \sum_{j=1}^d \sum_{x \in \mathbb{T}^d_N} \int_{\mathcal{X}^2_N} (h^{\ell,j}_x)^2 d\mu^N_t 
+ C K^2 e^{3 C_1 K } N^{d-1}
.\]
For the second term, noting that the random variables $\{ h^{\ell,j}_x \}$ are $(2 \ell-1)$-dependent, we decompose the summation $\sum_{x \in \mathbb{T}^d_N}$ into $\sum_{y \in \Lambda_{2\ell} } \sum_{z \in (4 \ell) \mathbb{T}^d_N \cap \mathbb{T}^d_N}$ and then apply the entropy inequality, which yields  
\[
\begin{aligned}
\sum_{x \in \mathbb{T}^d_N} \int_{\mathcal{X}^2_N} (h^{\ell,j}_x)^2 d\mu^N_t  
& \leq  \frac{1}{\gamma} \sum_{y \in \Lambda_{2 \ell}}  	\left( H(\mu^N_t | \nu^N_t ) + \log{  \int_{\mathcal{X}^2_N} \prod_{z \in (4 \ell) \mathbb{T}^d_N \cap \mathbb{T}^d_N} e^{\gamma (h^{\ell,j}_{z+y})^2} d\nu^N_t  }    \right)   \\
& = \frac{ (2\ell)^d }{\gamma} \left( H(\mu^N_t | \nu^N_t )  + \sum_{z \in (4\ell) \mathbb{T}^d_N \cap \mathbb{T}^d_N  } \log{ \int_{\mathcal{X}^2_N} e^{\gamma (h^{\ell,j}_{z+y})^2 } }  d\nu^N_t \right)
\end{aligned}
\]
for every $\gamma >0$. Moreover, recall here that by the flow lemma stated in Proposition \ref{flow lem} we can estimate the variance of $h^{\ell,j}_x$ as 
\[
\sigma^2 \coloneqq \sup_{x \in \mathbb{T}^d_N,  j = 1,...,d } {\rm Var}_{ \nu^N_t } [ h^{\ell,j}_x ]  \leq C_d g_d (\ell) e^{ 2 C_1 K }
\]
with $g_d(\ell)$ in Proposition \ref{flow lem}. Therefore, applying the concentration inequality, we have
\[
\log \int_{\mathcal{X}^2_N} e^{\gamma (h^{\ell,j}_x)^2} d\nu^N_t \leq  2
\]
for every $0 < \gamma \leq  C_0 \sigma^{-2}$. Therefore, by choosing $\gamma^{-1} = C_0^{-1} C_d g_d(\ell) e^{ 2 C_1 K } $, we have shown $E_{\mu^N_t } [ V -V^\ell ]$ is bounded above by
\[
\alpha N^2 \mathcal{D} ( \sqrt{ f^N_t }; \nu^N_t ) 
+ \frac{\bar{C}_\alpha \ell^d g_d(\ell) K^2 e^{ 5 C_1 K} }{ N } \left( H ( \mu^N_t  | \nu^N_t ) + \frac{N^d}{\ell^d}  \right)  
+ C K^2 e^{ 3 C_1 K } N^{d-1}
.\] 
Now recalling the growth rate of $K$ was slower than $\delta (\log N)^{1/2}$ by the assumption (A3)$_\delta$, we end the proof by choosing $\ell = N^{1/d -\kappa/d}$ when $d \geq 2$ and $\ell = N^{ 1/2-\kappa }$ when $d = 1$. 
\end{proof}

We thus estimated the cost to replace the reminder term $V$ to its local average $V^\ell$ and next we prove the following bound for $V^\ell$. 

\begin{lemma}
We assume the same conditions as Theorem \ref{prob part}. Then for any $\kappa > 0$, we have
\[
E_{\mu^N_t} \left[V^\ell \right]  \leq C K H( \mu^N_t | \nu^N_t ) + C_\kappa N^{d-1+\kappa}
\]
when $d \geq 2$. When $d = 1$, the last term on the right hand side of the above is replaced by $N^{1/2 + \kappa}$. 
\end{lemma}
\begin{proof}
We again decompose the sum $\sum_{x \in \mathbb{T}^d_N}$ in the definition of $V^\ell$ as $\sum_{y \in \Lambda_{2\ell}} \sum_{ z \in (4\ell) \mathbb{T}^d_N \cap \mathbb{T}^d_N }$ and recall $a_{y_0 + x, \Lambda + x}$ is uniformly bounded above by some $C_a > 0$. Then, by using the entropy inequality and the concentration inequality to show
\[
\begin{aligned}
\int_{\mathcal{X}^2_N} V^\ell d\mu^N_t  
&\le \frac{C_a K }{\gamma} \sum_{y \in \Lambda_{2\ell}} 
\bigg( H(\mu^N_t | \nu^N_t) + \sum_{z \in (4\ell) \mathbb{T}^d_N \cap \mathbb{T}^d_N} \log E_{\nu^N_t} [ e^{\gamma \overleftarrow{ ( \omega_1 ) } _{z+y,\ell } \overrightarrow{ ( \omega_2 ) }_{z+y, \ell}   } ]  \bigg)  \\
& \le \frac{ C_a K ( 4 \ell )^d }{\gamma} \bigg(  H(\mu^N_t | \nu^N_t) + \frac{N^d}{ (4\ell)^d } C_1 \gamma \ell^{-d}  \bigg)
\end{aligned}
\]
for $\gamma = c \ell^{d}$ with $c >0$ small enough. Then recalling the way to take $\ell$ when $d \geq 2$ and $d=1$, we have the desired bound and end the proof.
\end{proof}

Hence, we complete the proof of (\ref{est V_G}) for $V$ defined by (\ref{V}) involving all terms appearing in $V_G$.

\subsection{Proof of (\ref{est V_K})}
We now discuss the contribution of 
\[
V_K (t) = -\frac{N^2}{2} \sum_{x,y \in \mathbb{T}^d_N, | x-y | = 1} (u^N (x) - u^N (y) )^2 \omega_{1,x} \omega_{1, y} 
.\]
But this can be estimated in the same manner as \cite{DMFPV19} and \cite{FT18} as follows. We let
\[
V^\ell_K (t) \coloneqq - N^2 \sum_{x \in \mathbb{T}^d_N} \sum_{j=1}^d ( u^N (x) - u^N (x+e_j) )^2 
\overleftarrow{ ( \omega_1 ) }_{x, \ell} 
\overrightarrow{ ( \omega_1 ) }_{x+ e_j, \ell} 
.\]
Using the pointwise estimate for the spatial derivatives of $u^N (t,x)$ proved in Lemma \ref{grad u}, we see that $N^2 (u^N (x) - u^N (y) )^2 $ has order $K^2$ for every $x, y \in \mathbb{T}^d_N$ with $|x-y| = 1$. Therefore, repeating the same argument for $V_G$, we obtain the desired estimate (\ref{est V_K}) where $K$ in (\ref{est V_G}) is replaced by $K^2$.

\section{Several estimates on discrete reaction-diffusion system (\ref{dHDL eq})}
\label{sec:PDE}
In this section, we give some estimates for macroscopic quantities which are determined by solving the semi-discretized hydrodynamic limit equations (\ref{dHDL eq}). Throughout this section, let $u^N=\{ u^N(t,x) \}_{t \in [0,T], x \in \mathbb{T}^d_N}$ and $v^N=\{ v^N(t,x) \}_{t \in [0,T], x \in \mathbb{T}^d_N}$ be the non-negative solution of (\ref{dHDL eq}). First we show the following comparison principle in general form under our discrete settings. 

\begin{lemma}
\label{comparison}
Let $( t, u ) \mapsto f(t , x, u)$ be a real-valued smooth function on $[0, T] \times \mathbb{R}^{ \mathbb{T}^d_N } $ for every $x \in \mathbb{T}^d_N$. Let $u^N(t,x)$ be a unique solution of 
\begin{equation}
\label{comparison}
\partial_t u^N (t,x) = \Delta^N u^N(t,x) + f (t, x, u^N(t) )
\end{equation}
and let $\overline{u}^N (t,x)$ (resp. $\underline{u}^N (t,x)$) be a super- (resp. sub-) solution. Namely, $ \overline{u}^N $ (resp. $\underline{u}^N$) satisfies \eqref{comparison} with ``$\ge$'' (resp. ``$\le$'') instead of the equality. Then we have $u^N (t,x) \leq \overline{u}^N (t,x) $ (resp. $u^N (t,x) \geq \underline{u}^N (t,x)$) for every $t \in [0,T]$ and $x \in \mathbb{T}^d_N$ provided $u^N (0,x) \le \overline{u}^N (0,x) $ (resp. $u^N(0, x) \ge \underline{u}^N (0, x) $) for every $x \in \mathbb{T}^d_N$. 
\end{lemma}
\begin{proof}
We give the proof only for super-solution since it can be proved in the same manner for sub-solution. Let $\overline{u}^N(t,x)$ be any given super-solution, that is, it satisfies
\[
\partial_t \overline{u}^N (t,x) \ge \Delta^N \overline{u}^N(t,x) + f (t, x, \overline{u}^N(t) )
\] 
for every $t \in [0,T]$ and $x \in \mathbb{T}^d_N $ by definition. Then, subtracting \eqref{comparison} on both side of the above display to obtain 
\begin{equation}
\begin{aligned}
\label{difference}
\partial_t ( \overline{u}^N (t,x) - u^N (t,x) ) \ge 
& \Delta^N ( \overline{u}^N (t , x ) - u^N ( t,x)  ) \\
& \quad + \tilde{f}( t, x, \overline{u} (t) , u (t) ) ( \overline{u}^N (t,x) - u^N (t,x)  ) .
\end{aligned}
\end{equation}
Here, $\tilde{ f } = \tilde{f} (t, x, \overline{u}^N (t) , u^N (t) )$ is defined by 
\[
\tilde{f} (t, x, \overline{u}^N (t) , u^N (t) ) = 
\begin{cases}
 \displaystyle\frac{ f ( t,  x, \overline{u}^N (t ) ) - f (t, x, u^N (t)) }{  \overline{u}^N (t,x) - u^N (t,x) } &\text{if } \overline{u}^N (t,x)  \neq u^N (t,x) ,  \\
\displaystyle\frac{ \partial f }{  \partial u (x) } ( t, x, u  ) \bigg|_{ u  = u^N (t)  }   &\text{if } \overline{u}^N (t,x) = u^N (t,x) .
\end{cases}
\]
Let $M \coloneqq \sup_{ (t, x ) \in Q_T }  | \tilde{ f } (t, x, \overline{u}^N (t) , u^N (t ) ) | $ and let $w^N (t,x) \coloneqq ( \overline{u}^N (t,x) - u^N (t,x) ) e^{ M t } + 2 \varepsilon - \varepsilon e^{-t}$ with $\varepsilon > 0$. Note here that such $M < \infty$ exists since $\overline{u}^N (t, x)$ and $u^N (t, x)$ are both continuous in $t$ for every $x \in \mathbb{T}^d_N$, and also by the assumption for the initial function we have $w^N (0, x) > 0$ for every $x \in \mathbb{T}^d_N$. In the sequel, we show $w^N \geq 0 $ in $[0,T] \times \mathbb{T}^d_N$ by contradiction. Suppose there exists a point $(t_0, x_0 ) \in (0,T] \times \mathbb{T}^d_N$ such that $w^N(t_0, x_0) = 0$ for the first time and $w^N (t,x) > 0$ for every $t \in [0, t_0)$ and $x \in \mathbb{T}^d_N$. Then, since $(t_0, x_0 )$ attains minimum of $w^N$ in $[0, t_0] \times \mathbb{T}^d_N$, we have $\partial_t w^N (t_0, x_0) \leq  0$ and $\Delta^N w^N (t_0, x_0) \geq 0$ and thus $\partial_t w^N (t_0, x_0) - \Delta^N w^N (t_0, x_0) \le 0$. 
On the other hand, letting $\tilde{u}^N \coloneqq \overline{u}^N - u^N$, we have by definition of $M$
\[
\begin{aligned}
& \partial_t w^N (t_0, x_0) - \Delta^N w^N (t_0, x_0) \\
& \, = \big(  \partial_t  \tilde{u}^N (t_0, x_0) - \Delta^N \tilde{u}^N (t_0, x_0)  + M \tilde{u}^N (t_0, x_0) \big) e^{ M t_0 } + \varepsilon e^{-t_0 }   \\
& \, \ge \big(  \partial_t  \tilde{u}^N (t_0, x_0) - \Delta^N \tilde{u}^N (t_0, x_0)  - \tilde{f} (t_0,  x_0 , \overline{u}^N (t_0) , u^N (t_0) ) \tilde{u}^N (t_0, x_0) \big) e^{ M t_0 } + \varepsilon e^{-t_0 } .
\end{aligned}
\]
However, since $\overline{u}^N$ is a super-solution of \eqref{comparison}, the estimate \eqref{difference} at the point $(t_0, x_0)$ implies that the last quantity is bounded from below by a strictly positive constant, which is contradiction. Therefore, we have $w^N \ge 0$ so that $\overline{u}^N (t, x ) - u^N (t, x )  \ge \varepsilon ( e^{-t } - 2 ) e^ {- M t }$ for every $(t, x ) \in Q_T $. Since $\varepsilon > 0$ was taken arbitrary, we complete the proof by letting $\varepsilon$ tends to zero. 
\end{proof}

Since our exclusion rule prohibits same kind of particles to stay on the same site, density of each particles would not leave the interval $[0,1]$. Following two lemmas ensure this intuition and give some quantitative estimates of densities from below and above.

\begin{lemma}
\label{dsol est}
For every $t \in (0,T]$ and $x \in \mathbb{T}^d_N$, we have 
\[
0 \leq u^N(t,x),v^N(t,x) \leq 1
\]
provided $0 < u^N(0,x),v^N(0,x) < 1$ holds for every $x \in \mathbb{T}^d_N$.
\end{lemma}
\begin{proof}
First we observe that the zero function $u^N (t,x) \equiv 0 $ satisfies the first equation of \eqref{dHDL eq}. In particular, the function $0$ is a sub-solution of the first equation of \eqref{dHDL eq} so that for any solution $u^N (t,x)$ of the first equation of \eqref{dHDL eq} we have $u^N ( t, x ) \ge 0$ for every $t \in [0, T]$ and $x \in \mathbb{T}^d_N$ according to the comparison principle (Lemma \ref{comparison}). On the other hand, viewing the second equation of \eqref{dHDL eq}, we can solve it for $v^N$ explicitly as
\[
v^N (t, x) = v^N_0 (x) e^{ -K \int_0^t c_{2,x}(u^N( \tau ) ) u^N(\tau, x)  d\tau }
.\]
Since $u^N $ is proved to be non-negative, this explicit formula for $v^N$ implies that $v^N $ is non-decreasing in time and non-negative, which end the proof of the assertion $0 \le v^N (t,x) \le 1$ for every $t \in (0, T] $ and $x \in \mathbb{T}^d_N$. Finally, we observe that the constant function $u^N \equiv 1$ satisfies the first equation of \eqref{dHDL eq} with ``$\ge$'' instead of the equality so that it becomes to be a super-solution. Therefore, combining with the non-negativity of any solution $u^N $, we have the assertion for $u^N$ again by the comparison principle and complete the proof. 
\end{proof}

Next we let $M_i \coloneqq \sup_{ (\rho_1, \rho_2) \in [0,1]^{ \mathbb{T}^d_N } \times [0,1]^{ \mathbb{T}^d_N } } c_i (\rho_1, \rho_2)  \geq 0$ for $i = 1,2$, which are independent of the scaling parameter $N$.

\begin{lemma}
\label{dsol est2}
If there exists a positive constant $\delta_i \in (0,1)$  $(i=1,2)$  such that $ \delta_1 < u^N (0,x) < 1- \delta_1 $ and $\delta_2 < v^N (0, x) < 1- \delta_2$ for all $x \in \mathbb{T}^d_N$, then we have
\[
\delta_1 e^{-K M_1 t} \leq u^N (t,x) \leq 1-\delta_1 , \quad 
\delta_2 e^{-K M_2 t} \leq v^N (t,x) \leq 1- \delta_2
\] 
for every $t \in [0,T]$ and $x \in \mathbb{T}^d_N$. 
In particular, if $0 < u^N (0,x), v^N (0,x) < 1$ for every $x \in \mathbb{T}^d_N$, then $ 0 < u^N(t,x),v^N (t,x) < 1 $ for every $t \in [0,T]$ and $x \in \mathbb{T}^d_N$.
\end{lemma}
\begin{proof}
The assertion for $v^N$ is obvious from its explicit representation given in the previous lemma so we show the assertion only for $u^N$. To prove the lower bound for $u^N$, let us define $\underline{u} (t) \coloneqq \delta_1 e^{-K M_1 t}$ and $ w^N (t, x) \coloneqq ( u^N(t,x ) - \underline{u} (t) ) e^{- 2 K M_1 t}$. Then, since $0 \le u^N(t,x), v^N(t,x) \leq 1$ by Lemma \ref{dsol est}, we have
\[
\begin{aligned}
\partial_t w^N  & = \Delta^N  w^N   - K \left[  c_{1,x} (u^N(t), v^N(t) ) u^N(t,x) v^N(t,x) - M_1 \underline{u}(t)  \right] e^{- 2 K M_1 t} - 2K M_1w^N  \\
& \ge \Delta^N w^N - K \left[ M_1 u^N (t,x) v^N (t,x) - M_1 \underline{u} (t) \right] e^{ - 2 K M_1 t } - 2 K M_1 w^N \\ 
& \ge \Delta^N  w^N  - 3 K M_1 w^N 
\end{aligned}
\]
Since $w^N (0, x) > 0$ for every $x \in \mathbb{T}^d_N$ by definition, we have $w^N \geq 0$ in $[0,T] \times \mathbb{T}^d_N$ by Lemma \ref{comparison}. The upper bound is obvious from Lemma \ref{comparison}. 
\end{proof}

Next we give a priori estimates for (\ref{dHDL eq}) which are needed to prove the relative compactness of the sequence of discrete solutions. 
Let $p^N(t,x,y)$ be the discrete heat kernel of $\Delta^N$ on $\mathbb{T}^d_N$. Then, we have the following estimate.

\begin{lemma}
\label{dheat ker}
There exist positive constants $C, c>0$ such that
\[
|\nabla^N p^N(t,x,y)| \leq C p^N(ct, x,y) / \sqrt{t} 
\] 
for all $t>0$ and $x, y \in \mathbb{T}^d_N$.
\end{lemma}

This lemma is already shown as Lemma 2.6 in \cite{DMFPV19} or Lemma 4.2 in \cite{FT18} so we omit the proof here. Using this estimate for the discrete heat kernel, we obtain the following pointwise estimate for growth of derivatives of discrete solution through the same manner as \cite{DMFPV19}. 

\begin{lemma}
\label{grad u}
The gradients of the solution $u^N(t,x)$ of (\ref{dHDL eq}) are estimated as 
\[
|\nabla^N u^N(t,x)| \leq K (C_0 + C \sqrt{t})
\]
for every $t>0$ and $x \in \mathbb{T}^d_N$ if $|\nabla^N u^N(0,x)| \leq C_0 K$ holds for every $x \in \mathbb{T}^d_N$. 
\end{lemma}
\begin{proof}
By Duhamel's principle applied to the first equation of (\ref{dHDL eq}), we have
\begin{align*}
u^N(t,x) 
&= \sum_{y \in \mathbb{T}^d_N} u^N(0,y) p^N(t,x,y)  \\
&\quad - K(N) \int_0^t \sum_{y\in \mathbb{T}^d_N} c_{1, y } ( u^N(t), v^N(t) ) u^N(t,y) v^N(t,y) p^N(t-s, x,y) ds
\end{align*}
for every $t \in [0,T]$ and $x \in \mathbb{T}^d_N$. Noting that the reaction rate $c_1$ is assumed to be bounded, the absolute value of the gradient $\nabla^N u^N(t,x)$ can be bounded above by 
\[
\sum_{y \in \mathbb{T}^d_N} |\nabla^N u^N(0,y)| p^N(t,x,y) + K(N) M_1 \int_0^t \sum_{y \in \mathbb{T}^d_N} | \nabla^N p^N(t-s,x,y) |ds
\]
and thus we complete the proof in view of the assumption (A1) and $\sum_{y} p^N (t,x, y ) = 1$ for every $t, x$ for the first term and Lemma \ref{dheat ker} for the second term. 
\end{proof}

\begin{lemma}
\label{react int}
We have that
\[
\sup_{N\in \mathbb{N}} \int_0^T  \frac{1}{N^d} \sum_{x \in \mathbb{T}^d_N} K(N)  c_{1, x} ( u^N(t), v^N(t)) u^N(t,x) v^N(t,x) dt 
\le 1
.\]
\end{lemma}
\begin{proof}
From the first equation of (\ref{dHDL eq}), integrating over $t \in [0,T]$ and $x \in \mathbb{T}^d_N$ to represent the integration of reaction term by terms which are independent of $K(N)$. Since summation over $x \in \mathbb{T}^d_N$ of $\Delta^N u^N(t,x)$ vanishes and the term involving time derivative becomes an integration on the boundary, the proof is obvious in view of Lemma \ref{dsol est}.
\end{proof}

\begin{lemma}
\label{energy u}
We have that 
\[
\sup_{N \in \mathbb{N}} \int^T_0 \frac{1}{N^d} \sum_{x \in \mathbb{T}^d_N}  |\nabla^N u^N(t,x)|^2dt \leq \frac{1}{2}  
.\]
\end{lemma}
\begin{proof}
Multiplying $u^N(t,x)$ on both sides of the first equation of (\ref{dHDL eq}) and summing up over $x \in \mathbb{T}^d_N$, we obtain
\begin{align*}
\frac{1}{2} \frac{\partial}{\partial t} \frac{1}{N^d} \sum_{x \in \mathbb{T}^d_N} u^N(t,x)^2 
+ \frac{1}{N^d} \sum_{x \in \mathbb{T}^d_N} |\nabla^N u^N(t,x)|^2  \\
= -\frac{K(N)}{N^d} \sum_{x \in \mathbb{T}^d_N} c_{1, x} ( u^N (t), v^N(t) ) u^N(t,x)^2 v^N(t,x) 
\end{align*}
for every $t \in[0,T]$. Since the right hand side of this equation is always non-positive, this further implies
\[
\int_0^T \frac{1}{N^d} \sum_{x \in \mathbb{T}^d_N} |\nabla^N u^N(t,x)|^2 dt \leq \frac{1}{2N^d} \sum_{x \in \mathbb{T}^d_N}  \left[ u^N(0,x)^2 - u^N(T,x)^2 \right] \leq \frac{1}{2}.
\]
\end{proof}

\section{Case 1: Vanishing interface}
\label{sec:case1}
In this section we consider the semi-discretized system  
\begin{equation}
\label{case1}
\begin{cases}
\partial_t u^N(t,x) = \Delta^N u^N(t,x) - K(N) u^N (t,x) u^N (t,x+z_1)\cdots u^N(t, x+z_{m-1}) v^N(t,x) \\
\partial_t v^N(t,x) = - K(N) u^N (t,x) v^N(t,x) 
\end{cases}
\end{equation}
where $t \in [0,T]$, $x \in \mathbb{T}^d_N$ and $z_i \in \mathbb{Z}^d, i=1,...,m-1$ with $m >3$. We show in the sequel that taking limit as $N$ tends to infinity $v^N(t,x)$ vanishes at any time $t > 0$ and $u^N(t,x)$ converges to a unique solution of the heat equation on the whole domain. Through this section, in addition to the assumptions (A1), (A2) and (A3), we further assume (B1) which ensures that the initial function $u (0,\cdot)$ has better regularity than other cases. This is used in order to approximate a solution of a linear hyperbolic equation \eqref{subsol eq} by solutions of semi-discretized version with a good rate as the scaling parameter $N$ tends to infinity.


\begin{theorem}
\label{case1 thm}
We assume (A1), (A3)$_\delta$ and (B1) with some $\delta > 0$. Let $\{ u^N (t, \theta) \}$ be defined by (\ref{dsol}) and let $u = u(t, \theta)$ be a solution of the heat equation (\ref{heat eq}) on the whole domain with periodic boundary condition. Then for every $t \in ( 0, T]$ we have 
\[
\lim_{ N \to \infty} \sup_{ x \in \mathbb{T}^d_N } | u^N (t, x) - u (t, x/N) | = 0, \quad 
\lim_{ N \to \infty} \sup_{ x \in \mathbb{T}^d_N } | v^N (t, x ) | = 0.
\]
\end{theorem}

Since the second equation of  (\ref{case1}) is liner for $v^N$, it suffices to study the limiting behavior of the single equation
\begin{equation}
\label{case1'}
\partial_t u^N(t,x) = \Delta^N u^N(t,x) - K(N) v^N_0 (x) u^N(t,x)\cdots u^N (t,x+z_{m-1}) e^{-K(N) \int_0^t u^N(\tau,x) d\tau} 
.\end{equation}
To prove Theorem \ref{case1 thm}, we construct the functions $\rho^N_\pm (t,x)$ such that 
\[
\rho^N_- (t,x) \leq u^N(t,x) \leq \rho^N_+(t,x) 
\]
hold for every $t \in [0,T]$, $x \in \mathbb{T}^d_N$ and both $\rho^N_+ $ and $\rho^N_-$ converges to the solution to the heat equation (\ref{heat eq}). 

First we construct a super-solution of \eqref{case1'} which bounds the solution $u^N$ from above. Let $\rho^N_+= \{ \rho^N_+(t,x) \}_{t \in [0,T], x \in \mathbb{T}^d_N}$ be the solution of the semi-discrete heat equation
\begin{equation}
\label{dheat eq}
\begin{cases}
\partial_t \rho^N_+(t,x) = \Delta^N \rho^N_+(t,x) \\
\rho^N_+(0,x)= u^N( 0, x)
\end{cases}
\end{equation}
Since $u^N (t,x)$ and $v^N (t,x)$ are supposed to be positive for all $t \in [0,T]$ and $x \in \mathbb{T}^d_N$, the reaction term of (\ref{case1'}) is always non-positive. Therefore $\rho^N$ is a super-solution of \eqref{case1'} and the comparison principle shown in Lemma \ref{comparison} assures that $u^N(t,x)$ is bounded above by $\rho^N_+ (t,x)$ for every $t \in [0,T]$ and $x \in \mathbb{T}^d_N$. Thus our remainder task is to construct the sub-solution which asymptotically satisfies the heat equation (\ref{heat eq}).

As we see below, one can find such a sub-solution as a same manner with \cite{IMMN17}. However, we have to rearrange the building procedure to fit our discrete setting. 
To construct the sub-solution $\rho^N_-$, we consider the following problem for each fixed constant $\delta>0$. Let $\underline{u}_\delta = \underline{u}_\delta (t, \theta)$ be a solution of 
\begin{equation}
\label{subsol eq}
\begin{cases}
\partial_t \underline{u}_\delta = \Delta \underline{u}_\delta - \delta \underline{u}_\delta \\
\underline{u}_\delta (0, \theta ) = u _0(\theta ) 
\end{cases}
\end{equation}
and let $\{  \underline{u}^N_\delta (t, x)  \}_{N \in \mathbb{N}}$ be a solution of 
\begin{equation}
\label{dsubsol eq}
\begin{cases}
\partial_t \underline{u}^N_\delta = \Delta^N \underline{u}^N_\delta - \delta \underline{u}^N_\delta \\
\underline{u}^N_\delta (0, x) = u^N (0,  x ) .
\end{cases}
\end{equation}
In fact, it becomes necessary to use that $\underline{u}_\delta$ can be approximated by $\underline{u}^N_\delta$ with a rate strictly faster than $K^2$ to construct a desired sub-solution. Such a result can be easily obtained by the convergence result of semi-discretized heat equation to the classical one as we see in the sequel. First we can find the convergence rate for heat equation as stated in \cite{JL06}.

\begin{proposition}[\cite{JL06}]
\label{heat aprx}
Assume that $\rho_0$ is a real-valued function on $\mathbb{T}^d$ with a bounded fourth derivative. Let $\{ \rho^N (t, x) \}_{N \in \mathbb{N}} $ be a solution to the semi-discretized heat equation
\[
\begin{cases}
\begin{aligned}
& \partial_t \rho^N (t, x) = \Delta^N \rho^N (t, x) \\
& \rho^N (0, x) = \rho_0 (x / N)
\end{aligned}
\end{cases}
\]
and let $\rho(t,\theta)$ be a solution of the heat equation
\[
\begin{cases}
\begin{aligned}
& \partial_t \rho (t, \theta) = \Delta \rho (t, \theta) \\
& \rho (0, \theta) = \rho_0 (\theta).
\end{aligned}
\end{cases}
\]
Then there exists a positive constant $C$ such that for every $t \ge 0$ and $ x \in \mathbb{T}^d_N$ we have 
\[
| \rho^N (t, x) - \rho (t, x / N )  | \le C t / N^{ 2 }.
\]
\end{proposition}

In other words, a solution to the (continuous) heat equation is approximated by that of the semi-discretized one with precision order $1/N^2$ if they have the common initial function. We see that $\underline{u}_\delta = e^{- \delta t } \rho$ and $\underline{u}^N_\delta = e^{ -\delta t } \rho^N$ where $\rho$ is a solution of the classical heat equation \eqref{subsol eq} with initial function $u_0$ and $\rho^N$ is a solution of semi-discretized heat equation \eqref{dsubsol eq} with initial function $\rho^N (0, \cdot) = u^N (0, \cdot)$. Moreover, recall here that their initial function is close up to the order $o_N (1/ K^{2} )$ by the assumption (B1). Therefore, we apply Proposition \ref{heat aprx} to obtain
\[
\lim_{ N \to \infty } K^2 | \underline{u}^N_\delta (t, x) - \underline{u}_\delta (t, x/N)  | = 0 
\] 
uniformly in $t \in [0,T] $. Namely, we can approximate the solution of (\ref{subsol eq}) by the solution of (\ref{dsubsol eq}) with precision order $o_N (1/K^2 )$. With these approximation results at hand, now we show the following two lemmas (Lemma \ref{small time} and \ref{subsol bound}) as preliminary to prove Theorem \ref{case1 thm}.

\begin{lemma}
\label{small time}
Assume (B1) and that let $m$ be an integer satisfying $m > 3$. Let $\underline{u}_\delta$ be a solution of (\ref{subsol eq}). Then there exists a positive constant $t_* = t_* (\delta)$ such that 
\begin{equation}
\label{small time est}
\left[ (m-1) \underline{u}_\delta (t, \theta )^{m-3}  \partial_t \underline{u}_\delta (t, \theta) -1 \right] \underline{u}_\delta (t,\theta ) \leq 0 
\end{equation}
for every $t \in [0,t_*]$ and $\theta \in \text{supp} v_0$. 
\end{lemma}
\begin{proof}
Since $\theta \in \text{supp} v_0$, the quantity inside the brackets in \eqref{small time est} is $-1$ at initial time $t = 0$. Therefore, by the continuity in time of $\underline{u}_\delta $, the assertion holds in a short time interval and thus we complete the proof.  
\end{proof}

The above assertion is a result not in a discrete setting but completely in the PDE context and of course the time horizon $t_*$ is independent of the scaling parameter $N$. 


\begin{lemma}
\label{subsol bound}
We assume the same conditions as Theorem \ref{case1 thm}. Let $\delta> e^{-1}$ and $t_* > 0$ be the constant given in Lemma \ref{small time}. Then for every sufficiently large $N \in \mathbb{N}$ we have 
\begin{equation*}
u^N(t,x) \geq \underline{u}^N_\delta (t,x)
\end{equation*}
for every $t \in [0,t_*]$ and $x \in \mathbb{T}^d_N$. 
\end{lemma}

\begin{proof}
Let $\varepsilon \in (0, \varepsilon_N)$ be given and $\varepsilon_N$ will be characterized later. Let $W^N \coloneqq u^N - \underline{u}^N_\delta + \varepsilon$. Then $W^N$ satisfies 
\[
\begin{cases}
\partial_t W^N (t,x) = \Delta^N W^N (t,x)  + I^N (t,x) \\
W^N(0,x) = \varepsilon 
\end{cases}
\]
with 
\[
I^N (t, x ) \coloneqq - K v^N_0 (x) u^N(t,x) \cdots u^N(t,x+z_{m-1}) e^{-K \int_0^t u^N(\tau , x) d\tau } + \delta \underline{u}^N_\delta (t, x)  
\]
for every $t \in [0, T]$ and $x \in \mathbb{T}^d_N$. We show $W^N > 0$ in $Q_{t_*}$ by contradiction. To see that, suppose there exists a $(t_0, x_0) \in Q_{t_*} $ such that 
\[
W^N (t_0, x_0 ) = 0, \quad W^N (t,x ) > 0 \text{ in } [0, t_0 ) \times \mathbb{T}^d_N 
.\]
Let $I_0^N = I^N (t_0, x_0) $ and our task is to prove $I^N_0 > 0$. Indeed, since $W^N$ attains its minimal value $0$ at the point $(t_0, x_0)$ in $Q_{t_0}$, we have
\[
\partial_t W^N (t_0, x_0) \leq 0 , \quad \Delta^N W^N (t_0, x_0) \geq 0
.\]
However, once we proved $I^N_0 > 0$, we have 
\[
0 \geq \partial_t W^N (t_0, x_0) = \Delta^N W^N (t_0, x_0) + I^N_0 > 0
\]
which becomes contradiction. First we consider the case $x_0 / N \not\in \text{supp} v_0$. Since $u^N$ is non-negative and bounded from above by $1$, we have 
\[
\begin{aligned}
I^N_0 
& = - K v^N_0 (x_0) u^N( t_0, x_0 ) \cdots u^N(t_0, x_0 + z_{m-1}) e^{-K \int_0^{t_0} u^N (\tau , x_0 ) d\tau } + \delta \underline{u}^N_\delta (t_0 , x_0 )  \\
& \ge - K v^N_0 ( x_0 ) u^N (t_0, x_0 ) + \delta \underline{u}^N_\delta (t_0, x_0) .
\end{aligned}
\]
However, by the assumption (B1), we have the bound $K v^N_0 (x_0) \le \delta $ for every sufficiently large $K$ so that the last display is bounded form below by $- \delta (u^N (t_0, x_0 ) - \underline{u}^N_\delta (t_0, x_0) ) = \delta \varepsilon $ recalling $W^N (t_0, x_0) = u^N (t_0, x_0) - \underline{u}^N_\delta  (t_0, x_0 )  + \varepsilon =0$. Therefore, we have $I^N_0 > 0$ in the case of $x_0 / N \notin \text{supp} v_0$ so we assume $x_0 / N \in \text{supp} v_0$ in the sequel. To see $I^N_0 > 0$ in this case, we decompose $I^N_0 = I^N_1 + I^N_2 + I^N_3$ where
\[
\begin{aligned}
& I^N_1 = -K v^N_0 (x) \left( u^N (t_0, x_0) \cdots u^N (t_0, x_0 + z_{m-1}) - \underline{u}^N_\delta (t_0, x_0)^m  \right)  e^{-K \int_0^{t_0} u^N (\tau , x_0) d\tau}  \\
& I^N_2 = \underline{u}^N_\delta (t_0, x_0)  \left( \delta - K v^N_0 (x_0) \underline{u}^N_\delta (t_0, x_0)^{m-1} e^{-K \int_0^{t_0} \underline{u}_\delta^N (\tau , x_0) d\tau}  \right)  \\
& I^N_3 = K v^N_0 (x_0) \underline{u}^N_\delta (t_0, x_0)^m \left( e^{-K \int_0^{t_0} \underline{u}^N_\delta (\tau, x_0) d\tau } - e^{ -K \int_0^{t_0} u^N (\tau, x_0 ) d\tau } \right) 
\end{aligned}
\]
and estimate $I^N_1$, $I^N_2$ and $I^N_3$ separately. First for $I^N_1$, we replace the local product of $u^N$ into the spatially homogeneous one, that is, we prove
\[
\lim_{N \to \infty} \sup_{x \in \mathbb{T}^d_N } K | u^N (t, x) \cdots u^N (t, x+z_{m-1}) - u^N (t,x)^m | = 0
\]
for every $t \in [0,T]$. To see this, let $\{ y^j_i \}_{j = 0,...,| z_{i+1} - z_i |}$ be one of shortest paths from $z_i$ to $z_{i+1}$ for every $i= 0,...,m-2$: $y_i^0 = z_i$, $y_i^{| z_{i+1}- z_i | }= z_{i+1}$, $ | y_i^{j+1} - y_i^j | =1$ for every $i=0,...,m-2, j=0,..., | z_{i+1} - z_i | -1$ and we let $z_0 = 0$. Then, since $u^N (t,x)$ takes values in $[0,1]$ according to Lemma \ref{dsol est2}, the absolute value appearing in the left hand side is bounded above by 
\[
K \sum_{i = 0}^{m-2} \sum_{j=0}^{ | z_{i+1} - z_i | -1 } | u^N (t, y_i ^{j+1}) -u^N (t, y_i^j)  | = O ( K^2 / N )
.\]
Here we have used the pointwise estimate of derivatives $|\nabla^N u^N (t,x)|$ stated in Lemma \ref{grad u}. In particular, we can replace the spatially inhomogeneous local product into the homogeneous one and thus we have
\[
I^N_1 = K v^N_0 (x_0) \left( \underline{u}^N_\delta (t_0, x_0 )^m - u^N (t_0, x_0 )^m \right) e^{ -K \int_0^{t_0} u^N (\tau, x_0) d\tau } +O (K^2/ N)
.\]
However, by an elementary estimate $(u+\varepsilon )^m - u^m \geq m \varepsilon u^{m-1}$ for every $u, \varepsilon \geq 0$ and $m > 1$, the first term in the above comes out to be non-negative. Next we estimate $I^N_2$. Since $\underline{u}_\delta$ can be approximated by $\underline{u}^N_\delta $with precision order $o_N (1/K^2 )$ with help of Proposition \ref{heat aprx} and the assumption (B1), we can replace $\underline{u}^N_\delta$ inside parentheses in $I^N_2$ by $\underline{u}_\delta$ with a cost of order $o_N (K^2 / K^2 ) = o_N (1) $ which is small as $N$ tends to infinity. Here for $K^2$ in the numerator, one $K$ is the coefficient of the leading term (the first term in the above) and the second one comes from the exponent in the exponential term. Now we let $z_K \coloneqq K \underline{u}_\delta (t_0, x_0 / N)^{m-1}$. Then we have
\[
\begin{aligned}
I^N_2 
&= \underline{u}^N_\delta (t_0, x_0  ) \big( \delta - v^N_0 (x_0  ) z_K e^{-z_K} e^{z_K - K \int_0^{t_0} \underline{u}_\delta (\tau, x_0 / N ) d\tau }  + o_N (1) \big)  \\
& \geq \underline{u}^N_\delta (t_0, x_0  ) \big( \delta - v^N_0 (x_0  ) e^{-1} e^{z_K - K \int_0^{t_0} \underline{u}_\delta (\tau, x_0 / N ) d\tau }  + o_N (1) \big)  . 
\end{aligned}
\]
Looking the exponential term appearing in the last quantity, one can observe
\begin{equation}
\label{zk}
\begin{aligned}
& z_K - K \int_0^{t_0} \underline{u}_\delta (\tau , x_0 / N ) d\tau  = K \int_0^{t_0} \left[ (m-1) \underline{u}_\delta (\tau, x_0 / N )^{m-3} \partial_t \underline{u}_\delta (\tau, x_0 / N ) - 1 \right] \underline{u}_\delta (\tau, x_0 / N ) d\tau 
\end{aligned}
\end{equation}
since $\underline{u}_\delta (0, x_0 / N ) = u_0 ( x_0 ) = 0$ when $x_0 / N \in \text{supp} v_0$. Therefore, \eqref{zk} stays non-positive if $t_0 \leq t_*$ recalling $t_*$ is the small time horizon found in Lemma \ref{small time}. Therefore, $I^N_2$ can be bounded from below as
\[
I^N_2 \geq \underline{u}^N_\delta (t_0, x_0  ) \big( \delta - v^N_0 (x_0) e^{-1}  + o_N (1)  \big)
\]
as $N $ tends to infinity. Finally, for $I^N_3$, recalling the temporal assumption $\underline{u}^N_\delta (t, x) - u^N (t,x) \le \varepsilon$ in $[0,t_0] \times \mathbb{T}^d_N$, we have
\[
\begin{aligned}
I^N_3 
& = K v^N_0 (x_0) \underline{u}^N_\delta (t_0, x_0)^m e^{-K \int_0^{t_0} \underline{u}^N_\delta (\tau, x_0) d\tau} \big( 1 - e^{ K \int_0^{t_0} ( \underline{u}^N_\delta (\tau, x_0)- u^N (\tau, x_0) ) d\tau}  \big) \\
& \geq K v^N_0 (x_0) \underline{u}^N_\delta (t_0, x_0)^m e^{-K \int_0^{t_0} \underline{u}^N_\delta (\tau, x_0) d\tau} \left( 1- e^{K \varepsilon t_0} \right). 
\end{aligned}
\]
Combining all estimates obtained above and recalling $t_0 \leq t_*$, we conclude
\[
I^N_0 \ge \underline{u}^N_\delta ( t_0, x_0  )  \big( \delta - e^{-1} + o_N (1) \big) - K (e^{K \varepsilon t_*} - 1) + O (K^2 / N)
\]
as $N$ tends to infinity. We note here that we took $\delta > e^{-1}$ and that $\underline{u}^N_\delta $ is bounded from below by $e^{- C K} $ and we have $\lim_{ N \to \infty} e^{C K } K^2 /N = 0$ for every $\delta $ appearing in the assumption (A3). Therefore, we choose $ \varepsilon_N $ so small that the above quantity stays strictly positive for every fixed (but sufficiently large) $N$ and thus we complete the proof by showing contradiction.
\end{proof}

Now we construct a desired sub-solution and give the proof of Theorem \ref{case1 thm}.

\begin{proof}[Proof of Theorem \ref{case1 thm}]
Recall that $u^N(t,x)$ satisfies the single equation (\ref{case1'}). We first show that the reaction term in (\ref{case1'}), which is denoted by $J^N (t,x)$, converges to $0$ as $N$ tends to infinity. To see that, fix any $t \in ( 0, t_* ]$. Then, according to Lemma \ref{subsol bound}, we have
\[
\begin{aligned}
- J^N (t,x) 
& =  K v^N_0 (x) u^N(t,x)\cdots u^N (t,x+z_{m-1}) e^{-K(N) \int_0^t u^N(\tau,x) d\tau} \\
& \leq  K u^N(t,x) e^{- K \int_0^t \underline{u}^N_\delta (\tau, x)  d\tau }  \\
& =   K u^N(t,x) e^{- K \int_0^t ( \underline{u}^N_\delta (\tau,x) - \underline{u}_\delta (\tau,x/N) ) d\tau } e^{- K \int_0^t \underline{u}_\delta (\tau, x/ N ) d\tau }. 
\end{aligned}
\]
Here we did not replaced $u^N(t,x)$ by $1$ to use Lemma \ref{subsol bound} later again. We have seen that $\underline{u}_\delta $ is approximated by $\underline{u}^N_\delta$ with precision of order $o_N ( 1/ K^2 ) $. In particular, we have $\lim_{ N \to \infty } \sup_{x \in \mathbb{T}^d_N } K | \underline{u}^N_\delta (t, x) - \underline{u}_\delta (t, x /N) | = 0 $ for every $t \in [0,T]$ so that there exists a positive constant $C$ such that the last quantity is bounded above by 
\[
 C K u^N (t, x) e^{ - K \int_0^t \underline{u}_\delta (\tau, x / N ) d\tau } .
\] 
Moreover, by an elementary inequality $s^2 e^{-s} \leq 4e^{-2}$ for every $s \geq 0$, we have
\[
\begin{aligned}
- J^N (t,x) 
& \leq C K  \left( K e \int_0^t \underline{u}_\delta (\tau,x / N ) d\tau \right)^{-2} u^N(t,x) \\
& \leq \frac{C}{ K} \left(  \int_0^t \min_{\theta \in [0, 1)^d} \underline{u}_\delta (\tau, \theta) d\tau   \right)^{-2} u^N(t,x) . 
\end{aligned}
\]
We let $\gamma (t) \coloneqq \int_0^t \min_{\theta \in [0,1)^d} \underline{u}_\delta (\tau, \theta) d\tau$. Since $\underline{u}_\delta  (t, \theta) > 0$ for every $t > 0$ and $\theta \in [0,1)^d$, we have $\gamma( t ) > 0$ so that there exists a positive constant $K_* > 0$ such that $\gamma (t_*) = (K_*)^{-1/4}$. In the sequel we suppose $N$ is sufficiently large so that $K_* \le K (N) $. Then, since we have $\gamma(0) = 0$ and $t \mapsto \gamma (t)$ is a continuous, strictly increasing mapping, there exists $t_K \le  t_* $ such that  $\gamma (t_K) = K^{-1/4}$ and $t_{K(N)} \searrow 0$ as $N$ tends to infinity. Therefore, for every $t \in [t_K , t_*]$, we have
\[
- J^N (t , x) \le \frac{C}{ K  \gamma (t_K)^2 } u^N (t, x) = \frac{C}{ \sqrt{K}  } u^N (t,x) . 
\]
On the other hand, when $t \in ( t_*, T ]$, since the function $ t \mapsto \gamma (t) $ is increasing, by using a similar argument given above, we have 
\[
- J^N (t,x) \leq \frac{ C }{K \gamma(t_*)^2} u^N ( t , x )  \le \frac{ C }{\sqrt{K} } u^N(t,x) .
\]
Thus we proved 
\begin{equation}
\label{reaction}
 0 \geq J^N (t,x) \geq - \frac{ C }{\sqrt{K} } u^N(t,x) 
\end{equation}
for every $t \in [t_K,T]$ and $x \in \mathbb{T}^d_N$, which particularly implies $\sup_{ x \in \mathbb{T}^d_N} J^N (t , x )$ converges to $0$ as $N$ tends to infinity.

Now we construct a sub-solution $\rho^N_-$. Fix $\delta_1 > e^{-1}$. Then, by Lemma \ref{subsol bound}, we have
\[
u^N (t,x) \geq \underline{u}^N_{\delta_1} (t,x)
\]
for every $t \in [0, t_* ]$ and $x \in \mathbb{T}^d_N$. Next we let $\delta_2 = C/{\sqrt{K} }$. Then, according to the first step which is given above, the reaction term in \eqref{case1'} $J^N$ satisfies the bound \eqref{reaction} for every $t \in [t_K, T]$ and $x \in \mathbb{T}^d_N$ and thus $\underline{u}^N_{\delta_2} (t- t_K , x ; \underline{u}^N_{\delta_1}(t_K, \cdot; u^N_0) )$ becomes to be a sub-solution of \eqref{case1'}. Here $ \underline{u}_\delta^N (t, x ; u^N_0) $ denotes a solution of \eqref{dsubsol eq} with initial function $u^N_0 : \mathbb{T}^d_N \to \mathbb{R}$. Therefore, the comparison principle (Lemma \ref{comparison}) implies
\[
u^N (t,x) \ge \underline{u}^N_{\delta_2} (t- t_K , x ; \underline{u}^N_{\delta_1}(t_K, \cdot; u^N_0) )
\] 
for every $t \in [t_K, T]$ and $x \in \mathbb{T}^d_N$. Recalling $t_K \leq t_*$, we define
\[
\rho^N_- (t,x) \coloneqq
\begin{cases}
\begin{aligned}
&\underline{u}^N_{\delta_1}(t,x; u^N_0) && \text{ if } t \in [0,t_K], \\
&\underline{u}^N_{\delta_2} (t,x; \underline{u}^N_{\delta_1}(t_K, \cdot; u^N_0)) && \text{ if } t \in (t_K, T] .
\end{aligned}
\end{cases}
\]
Then $\rho^N_- (t,x)$ is continuous in $t$ and we have $u^N (t,x) \geq \rho^N_- (t,x)$ for every $t \in [0,T]$ and $x \in \mathbb{T}^d_N$. 

Next we show $ \lim_{ N \to \infty} \sup_{ x \in \mathbb{T}^d_N } | \rho^N_+ ( t, x ) - \rho^N_- (t, x) | = 0  $ for every $t \in [0, T]$ to obtain the result for $u^N$. We let $W^N = \rho^N_+ - \rho^N_-$ and $\overline{W}^N = \delta_1 t $ be functions defined in $Q_T$. Then we can easily see that $W^N$ satisfies 
\begin{equation}
\label{W}
\partial_t W^N = \Delta^N W^N + \delta_1 \underline{u}^N_{\delta_1} 
\end{equation}
in $Q_{t_K}$ and that $\overline{W}^N $ is a super-solution of \eqref{W}. Therefore, since $W^N$ and $\overline{W}^N$ have the same initial function, the comparison principle (Lemma \ref{comparison}) implies that $W^N \le \overline{W}^N $ in $Q_{t_K}$. In particular we have $W^N \le \delta_1 t_K $ in $Q_{t_K}$. Similarly, we can bound $W^N$ by a function $\delta_1 t_K + \delta_2 (t- t_K)$ in $Q_T \setminus Q_{ t_K }$. Combining these results, we obtain 
\[
\max_{ Q_T } W^N
\le \max_{ Q_{t_K} } W^N + \max_{ Q_T \setminus Q_{t_K} } W^N 
\le 2 \delta_1 t_K + T \delta_2 
\] 
and the last quantity converges to zero as $N$ tends to infinity recalling $ t_K \searrow 0$ and $\delta_2 $ is proportional to $K^{-1/2}$. Hence we have $\lim_{N \to \infty} W^N = 0$ and obtain the assertion for $u^N$. 

Finally, we show the assertion for $v^N$. To see that, we have
\[
v^N (t,x) = v^N_0 (x) e^{-K \int_0^t u^N(\tau,x) d\tau} 
\leq v^N_0 (x) e^{-K \int_0^{t \wedge t_*} u^N (\tau,x) d\tau} 
\leq C v^N_0 (x) e^{-K \int_0^{t \wedge t_*} \underline{u}_\delta ( \tau, x / N ) d\tau} 
.\]
The last term converges to $0$ as $N$ tends to infinity since for any $t > 0$ the function $\underline{u}_\delta (t, \cdot) $ is bounded from below by a strictly positive constant independent of $N$ and thus we complete the proof. 
\end{proof}

\section{Case 2: Moving interface}
\label{sec:case2}
For \textbf{Case 2}, our semi-discretized hydrodynamic limit system is given by 
\begin{equation}
\label{case2}
\begin{cases}
\partial_t u^N(t,x) = \Delta^N u^N(t,x) - K(N) u^N(t,x) v^N (t,x) v^N (t,x+z_1)\cdots v^N(t, x+z_{m-1}) \\
\partial_t v^N(t,x) = - K(N) u^N (t,x) v^N(t,x) 
\end{cases}
\end{equation}
for $t \in [0,T]$, $x \in \mathbb{T}^d_N$ and $z_i \in \mathbb{Z}^d, i=1,...,m-1$ with $m \geq 1$. In this section and the next section, we extend our semi-discretized functions $u^N=u^N(t,x)$ and $v^N=v^N(t,x)$ as a simple function on $[0,T] \times \mathbb{T}^d$ by (\ref{dsol}) and study limiting behavior of these extended functions. Looking the above semi-discretized reaction-diffusion system, the reaction term of the first equation contains the product of several $v^N$'s which are spatially dispersed. Since the diffusion coefficient for $v^N$ is zero, it seems that we may not be able to replace this product into the spatially homogeneous one. However, by the second equation of our system (\ref{case2}), we can see that derivatives of $v^N$ are controlled by those of $u^N$ and the initial function $v^N (0, \cdot)$, which enables us to conduct replacement procedure.

Limiting behavior of $u^N (t,\theta)$ and $v^N(t,\theta)$ as $N \to \infty$ is stated as follows.

\begin{theorem}
\label{case2 thm}
Assume (A1), (A3)$_\delta$ and (B2) with some $\delta > 0 $. Let $u^N(t,\theta)$ and $v^N(t,\theta)$ be defined by (\ref{dsol}). Then there exists functions $u$ and $v$ on $Q_T$ such that  
\begin{align*}
& u \in L^\infty(Q_T) \cap L^2 (0,T;H^1(\mathbb{T}^d)), \, v \in L^\infty (Q_T) \\
& 0 \leq u,v \leq 1 \text{ and } uv=0 \text{ a.e. in } Q_T, \\
& u^N \to u \text{ strongly in } L^2(Q_T)  \text{ and a.e. in } Q_T , \\
& (v^N)^m \rightharpoonup v^m \text{ weakly in } L^2(Q_T) 
\end{align*}
as $N$ tends to infinity. Moreover, $w:=u-v^{m} /m$ satisfies 
\begin{equation}
\label{weak Stefan}
-\int_{\mathbb{T}^d} (u_0 - v_0^m/m ) \varphi(0) d\theta + \iint_{Q_T} (-w \varphi_t + \nabla w_+ \cdot \nabla \varphi) d\theta dt =0 
\end{equation}
for all $\varphi \in H^1 (Q_T)$ such that $\varphi(T,\cdot) \equiv 0$.
\end{theorem}

The equation (\ref{weak Stefan}) is the weak formulation of the one-phase Stefan problem. As stated in \cite{IMMN17}, assuming the limiting interface is smooth and further $u$ and $v$ are smooth on their support, one can write (\ref{weak Stefan}) as a strong form (\ref{strong Stefan}). The problem (\ref{strong Stefan}) is the classical formulation of the one-phase Stefan problem with the latent heat $w_0|_{\Gamma(t)}/m$. Derivation of (\ref{strong Stefan}) from the weak form (\ref{weak Stefan}) can be done analogously to \cite{HHP00}.

Our plan to prove Theorem \ref{case2 thm} is as follows: first we show relative compactness of the sequence $\{ u^N(t,\theta) \}_{N \in \mathbb{N}}$ and $\{ v^N (t,\theta) \}_{N \in \mathbb{N}}$ so that they are convergent along a subsequence and then we show that any limit points along this subsequence satisfy the weak form of the one-phase Stefan problem (\ref{weak Stefan}). Moreover, according to the uniqueness of weak solution of one-phase Stefan problem, we can show that the above convergence holds for the full sequence.

Following this procedure, we first show that the sequence of discrete solutions $\{u^N (t,\theta) \}_{N \in \mathbb{N}}$ is relatively compact in $L^p(Q_T)$ for any $p \ge 2$ without any restriction on reaction rate $c_1$.

\begin{lemma}
\label{u relcpt}
We assume the same conditions as Theorem \ref{case2 thm}. Then the sequence $\{ u^N(t,\theta) \}_{N \in \mathbb{N}}$ is relatively compact in $L^p(Q_T)$ for any $p \ge 2$. 
\end{lemma}
\begin{proof}
In the sequel, we show that there exists a positive constant $C$ such that 
\begin{align*}
\int_0^{T-\tau} \int_{\mathbb{T}^d} | u^N(t+\tau,\theta) -u^N(t,\theta) |^p d\theta dt  \leq C \tau, \\
\int_0^T \int_{\mathbb{T}^d} |u^N(t,\theta+\alpha) - u^N(t,\theta) |^p d\theta dt \leq C | \alpha |
\end{align*}
for all $p \ge 2$, $\tau \in (0,T)$ and $\alpha \in \mathbb{R}^d$ sufficiently small. Once these estimates are proved, we complete the proof of lemma by the Fr\'{e}chet-Kolmogorov theorem (see for example \cite{B83}, Theorem \Rnum{4}.25 and Corollary \Rnum{4}.26).

First we show the equi-continuity along spatial direction with exponent $p= 1$. Once the case when $p =1$ is proved, then we obtain the assertion for any exponent $p \ge 1$ according to the uniform boundedness of $u^N$. Change of variables enables us to restrict our cases for non-negative $\alpha$. In this case, we observe 
\[
\begin{aligned}
& \iint_{Q_T} | u^N (t, \theta + n/N ) - u^N (t, \theta)| d\theta dt \le \frac{ n }{N} \iint_{ Q_T } | \nabla^N u^N (t, \theta) | d\theta dt , \\
& \iint_{ Q_T } | u^N (t, \theta + 1/r N ) - u^N (t, \theta) | d\theta dt \le \frac{1 }{r N } \iint_{ Q_T} | \nabla^N u^N (t, \theta) | d\theta dt  
\end{aligned}
\]
for every $n \in \mathbb{Z}_+ = \{ 0, 1,...\}$ and $r \ge 1$. Combining these two estimates and applying them for $\alpha = n / rN$ with $n = \lceil \alpha N \rceil$ and $r = \lceil \alpha N \rceil  / \alpha N $ to obtain
\[
\iint_{Q_T} | u^N (t, \theta + \alpha ) - u^N (t, \theta) | d\theta dt \le \alpha \iint_{Q_T } | \nabla^N u^N (t, \theta) | d\theta dt \le \alpha \| \nabla^N u^N \|_{ L^2 (Q_T)}
\]
for every $\alpha \ge 0$ where in the last estimate we used H\"{o}lder's inequality. According to the uniform energy estimate Lemma \ref{energy u}, we obtain the equi-continuity in spatial variables for any index $p \ge 1$. In particular, the second assertion holds for any $p \ge 2$.

Similarly, it suffices to prove the equi-continuity in time argument only for the case $p=2$ by again using the fact that $u^N $ is bounded uniformly in $N$. We remark here that when $1 \le p < 2$ another exponent for $\tau$ is needed so that we restrict our cases only for $p \ge 2$. The integral appearing in the left hand side of the first estimate for $p=2$ is equal to   
\[
\int_0^{T-\tau} \int_{\mathbb{T}^d} \left(  \int_0^\tau  \partial_t u^N(t+s,\theta)  ds \right)  \left( u^N(t+\tau,\theta) -u^N(t,\theta) \right)  d\theta dt 
.\]
However, using the first equation of (\ref{dHDL eq}) for the integrand, this quantity can be estimated from above by
\begin{align*}
&\int_0^\tau
 \left( \int_0^{T-\tau} \int_{\mathbb{T}^d} \left| \nabla^N u^N(t+s,\theta) \right|^2 d\theta dt \right)^{ 1 /2 }
 \left( \int_0^{T-\tau} \int_{\mathbb{T}^d} \left| \nabla^N u^N ( t + \tau ,\theta) \right|^2 d\theta dt \right)^{ 1/2 } ds \\
&+\int_0^\tau
 \left( \int_0^{T-\tau} \int_{\mathbb{T}^d} \left| \nabla^N u^N(t+s,\theta) \right|^2 d\theta dt \right)^{1/2}
 \left( \int_0^{T-\tau} \int_{\mathbb{T}^d} \left| \nabla^N u^N(t,\theta) \right|^2 d\theta dt \right)^{1/2} ds \\
&+2K \int_0^\tau \int_0^{T-\tau} \int_{\mathbb{T}^d} c_1(x, u^N(t+s), v^N(t+s) ) u^N(t+s,x) v^N(t+s,x) d\theta dt ds  
.\end{align*}
Here we used Schwarz's inequality to estimate the first and the second terms. Thus we get the desired estimate in view of Lemmas \ref{react int} and \ref{energy u}. 
\end{proof}

On the other hand, for the relative compactness of $v^N$, we only impose the following existence of a weakly convergent subsequence which is obvious from the uniform boundedness of $v^N $ in view of Lemma \ref{dsol est}. 

\begin{lemma}
\label{v relcpt}
We assume the same condition as Theorem \ref{case2 thm}. Then for any $p > 1$, the sequence $\{ v^N  \}_{ N \in \mathbb{N}}$ is weakly precompact in $L^p (Q_T)$. Namely, there exists a subsequence $(N_k)$ and $v \in L^p (Q_T)$ such that $v^{N_k } \rightharpoonup v$ weakly in $L^p (Q_T)$. 
\end{lemma}


\begin{proof}[Proof of Theorem \ref{case2 thm}]
For any $p >1$, by Lemma \ref{u relcpt} the sequences $\{ u^N(t,\theta) \}_{N \in \mathbb{N}}$ is strongly precompact in $L^p(Q_T)$, while by Lemma \ref{v relcpt} $\{ v^N(t,\theta) \}_{N \in \mathbb{N}}$ is weakly precompact in $L^p (Q_T)$. Therefore, there exists a subsequence $\{  N_k \}$ and functions $u, v \in L^p (Q_T)$ such that 
\[
\begin{aligned}
u^{N_k} \to u \text{ strongly in } L^p(Q_T), \quad v^{N_k} \rightharpoonup v \text{ weakly in } L^p(Q_T)
\end{aligned}
\]
for any $p > 1$. Moreover, by taking further subsequences if necessary (which again denoted by $N_k$), we see that $u^{N_k} \to u$ a.e. in $Q_T$. 
Next we show that the limit function $u$ belongs to $L^2 (0, T; H^1 (\mathbb{T}^d))$. For any test function $\varphi \in C^\infty (\mathbb{T}^d)$, $j=1,...,d$ and $t \in [0, T]$, we have 
\[
\int_{ \mathbb{T}^d } u^N (t, \theta) \partial^N_j \varphi (\theta ) d\theta = - \int_{ \mathbb{T}^d } \varphi (\theta ) \partial^N_j u^N (t, \theta)  d\theta 
\]
where $\partial^N_j $ is the discrete partial derivative on $j$-th direction defined by $\partial^N_j u (\theta) \coloneqq N \left[ u (\theta + e_j/N) - u (\theta) \right] $ for every $u : \mathbb{T}^d \to \mathbb{R}$. Taking limit along $(N_k)$ on the above identity, we see that $\partial^N_j u^N $ converges to the $j$-th partial derivative $\partial_j u$ in distributional sense for every $j=1,...,d$. Moreover, since $L^2 (\mathbb{T}^d )$-norm of the discrete derivative $\partial^N_j u^N (t, \cdot )$ is bounded above by some constant independent of $N$ in view of Lemma \ref{energy u}, $\partial_j u (t, \cdot)$ belongs to $L^2 (\mathbb{T}^d)$ for every $j = 1,...,d$ and thus we obtain $u \in L^2 (0,T; H^1 (\mathbb{T}^d))$. Moreover, by the second equation of (\ref{dHDL eq}), we have
\[
\iint_{Q_T} u^N(t,\theta) v^N(t,\theta) d\theta dt \leq \frac{1}{K}
\] 
for every $N \in \mathbb{N}$. Since $u^{N_k} \to u$ strongly in $L^2 (Q_T)$ and $v^{N_k} \rightharpoonup v $ weakly in $L^2(Q_T)$ as $k$ tends to infinity, their product $u^{N_k}v^{N_k}$ converges strongly in $L^1(Q_T)$ to $u v$. Therefore, taking limit along $N_k$ on the above bound, we get $uv=0$ a.e. in $Q_T$. 

Next we let $w^N:=u^N-(v^N)^m/m$ for every $N \in \mathbb{N}$. Note here that it is already shown that $w^N$ converges weakly to some $w$ along some subsequence $N_k$. We show that any limit point $w$ satisfies the weak form of the one-phase Stefan problem (\ref{weak Stefan}). To see this, we first rearrange the reaction term of the first equation of (\ref{case2}) to the homogeneous one, namely, we show for every $t \in [0,T]$ and $x \in \mathbb{T}^d_N$, the absolute value of the difference 
\[
Ku^N(t,x) v^N(t,x) \cdots v^N(t,x+z_{m-1}) - K u^N(t,x) v^N(t,x)^m 
\]
converges to $0$ as $N$ tends to infinity. For simplicity we may assume $m=2$ since for the case $m\geq 3$ it can be proven in a similar way. For $z_1 \in \mathbb{Z}^d$, let $\{ y_j \}_{j=0,...,|z_1|}$ be one of minimal paths from the origin $O$ of $\mathbb{Z}^d$ to $z_1$, namely $y_j \in \mathbb{Z}^d$, $y_0=O$, $y_{|z_1|}= z_1$ and $|y_{j+1}-y_{j}|=1$ for $j=0,...,|z_1|-1$. Then, by the triangle inequality, we have
\begin{align*}
&\left| K u^N(t,x) v^N(t,x) v^N(t,x+z_1) - K u^N(t,x) v^N(t,x)^2  \right|  \\
&\leq K \sum_{j=0}^{|z_1|-1} \left| v^N(t,x+y_{j+1})- v^N(t,x+y_{j})  \right|  \le \frac{K}{N} \sum_{j=0}^{|z_1|-1} \left| \nabla^N v^N(t,x+y_{j})  \right|
\end{align*}
for every $t \in [0,T]$ and $x \in \mathbb{T}^d_N$. Here we can see that the derivative of $v^N(t,x)$ has order $O(K)$ for $N \in \mathbb{N}$ since it was controlled by that of $u^N$ in view of the second equation of \eqref{case2}. Indeed, by the second equation of \eqref{case2}, we have for every $t \in [0,T]$ and $x \in \mathbb{T}^d_N $ that
\[
v^N (t, x) = v^N (0, x ) e^{ - K \int_0^t u^N (\tau, x) d\tau } ,  
\]
which implies that for every $j = 1,...,d$ the difference $v^N (t, x + e_j ) - v^N (t, x)$ is equals to 
\[
v^N (0,x + e_j ) \big( e^{ - K \int_0^t u^N (\tau,  x + e_j ) d\tau } - e^{ - K \int_0^t u^N (\tau, x ) d \tau }  \big) 
+ \big( v^N (0, x+ e_j) - v^N (0, x) \big)  e^{ - K \int_0^t u^N (\tau., x) d \tau }  . 
\]
Therefore, since the function $z \mapsto e^{ -z }$ is Lipschitz continuous on $[0, \infty)$, Lemma \ref{grad u} and the assumption (A1) assures that $\nabla^N v^N (t, x)$ has order $ O (K^2) $ as $ N $ tends to infinity. According to this derivative estimate for $v^N$, the above difference between the inhomogeneous product of $v^N$'s and the spatially-homogeneous one has order $O(K^3 / N)$, which converges to $0$ as $ N $ tends to infinity by the assumption (A3)$_\delta$. 

After rearranging the reaction term of the first equation of (\ref{case2}) to the homogeneous one, we subtract the second equation multiplied by $v^N(t,x)^{m-1}$ from the first equation to obtain 
\[
\partial_t w^N (t,x) = \Delta^N u^N(t,x) + O ( K^3 / N )
\] 
for every $t \in [0,T]$ and $x \in \mathbb{T}^d_N$ as $N$ tends to infinity. Test $\varphi \in H^1(Q_T)$ such that $\varphi(T,\cdot) \equiv 0$ and integrate over $Q_T$ to obtain the weak form (\ref{weak Stefan}) with $w^N$ instead of $w$ with an error term which vanishes as $N$ tends to infinity. Then we take limit along the subsequence $(N_k)$ to get a weak solution $w$. 

Finally, since it is known that the weak solution of (\ref{weak Stefan}) is unique (see \cite{HHP00}), above convergence occurs along the full sequence and thus we complete the proof of Theorem \ref{case2 thm}. 
\end{proof}

\section{Case 3: Immovable interface}
\label{sec:case3}
In this section we consider the semi-discretized reaction-diffusion system
\begin{equation}
\label{case3}
\begin{cases}
\partial_t u^N(t,x) = \Delta^N u^N(t,x) - K(N) u^N(t,x) v^N (t,x)  \\
\partial_t v^N(t,x) = - K(N) u^N (t,x) u^N (t, x+z_1) \cdots u^N (t, x+ z_{m-1} ) v^N(t,x)  
\end{cases}
\end{equation}
where $t \in [0,T]$, $x \in \mathbb{T}^d_N$ and $z_i \in \mathbb{Z}^d, i=1,...,m-1$ with $m > 1$. For every $N \in \mathbb{N}$, let $ u^N (t,\theta) $ and $v^N (t, \theta)$ be the macroscopic functions on $Q_T= [0,T]\times \mathbb{T}^d$ defined by (\ref{dsol}). Then we have the following convergence of $u^N (t,\theta)$ and $v^N(t,\theta)$ as $N$ tends to infinity.

\begin{theorem}
\label{case3 thm}
Assume (A1), (A3)$_\delta$ and (B3) with some $\delta > 0 $. Let $u^N =  u^N(t,\theta) $ and $v^N =  v^N(t,\theta) $ be defined by (\ref{dsol}). Then there exist subsequences $\{ u^{N_k}\}$ and $\{ v^{N_k} \} $ of $\{ u^N \}$ and $\{ v^N \}$, respectively, and $u$, $v$, $\zeta$ such that 
\begin{equation}
\label{case3 est1}
\begin{aligned}
&u, u^{m/2} \in L^\infty(Q_T ) \cap L^2 (0,T;H^1(\mathbb{T}^d)), v \in L^\infty (Q_T), \zeta \in H^{-1} (Q_T), \\
&0 \leq u \leq 1, 0 \leq v \leq 1, uv=0 \text{ a.e. in } Q_T, \\
&\zeta \geq 0 \text{ in } H^{-1} (Q_T), \quad u(0, \cdot) = u_0 (\cdot), \, v(0, \cdot) = v_0 (\cdot)
\end{aligned}
\end{equation}
and
\begin{align}
\label{case3 est2}
& u^{N_k} \to u,  (u^{N_k})^{m/2} \to u^{m/2} &&\text{ strongly in } L^p(Q_T) \text{ and a.e. in } Q_T, \\
\label{case3 est3}
&v^{N_k} \rightharpoonup v  &&\text{ weakly in } L^p(Q_T), \\
\label{case3 est4}
&{\left| \nabla^{N_k}  (u^{N_k})^{m/2} \right|}^2 \overset{*}{\rightharpoonup} \zeta  &&\text{ weakly}^* \text{ in } H^{-1} (Q_T) 
\end{align}
for any $p \ge 2$ as $N_k$ tends to infinity. Moreover, $u$, $v$ and $\zeta$ satisfy
\begin{equation}
\label{case3 weak}
\iint_{Q_T} \left\{ -\left( \frac{u^m}{m}-v \right) \varphi_t + \frac{2}{m} u^{\frac{m}{2}} \nabla u^{\frac{m}{2}} \cdot \nabla \varphi \right\} d\theta dt + \frac{4(m-1)}{m^2}  {}_{H^{-1}(Q_T)}\langle \zeta, \varphi \rangle_{H^1_0 (Q_T)}=0 
\end{equation}
for all $\varphi \in H^1_0 (Q_T)$. 
\end{theorem}

\begin{proof}
First we show (\ref{case3 est2}). By the similar argument given in Lemma \ref{u relcpt}, we can show the assertion for $u^N$ in view of Lemmas \ref{react int} and \ref{energy u} with help of the Fr\'{e}chet-Kolmogorov theorem and we further obtain $u \in L^2 (0,T; H^1 (\mathbb{T}^d) )$. On the other hand, the assertion for $(u^N)^{m/2}$ is proved in the same manner. Indeed, multiply $(u^N)^{m-1}$ on both sides of the first equation of (\ref{case3}) to get 
\[
 \partial_t (u^N )^m/m = (u^N )^{m-1} \Delta^N u^N - K (u^N )^m v^N .
\]
Then, integrating the above identity over $Q_T$, integration by parts enables us to calculate
\begin{equation}
\label{IBP}
\begin{aligned}
& \iint_{Q_T} \nabla^N (u^N )^{ m - 1 } \cdot  \nabla^N u^N  d\theta dt \\
& = \frac{1}{m}  \int_{ \mathbb{T}^d } \big( u^N (0)^m - u^N ( T )^m \big) d\theta  - \iint_{Q_T} K (u^N )^m  v^N d\theta dt 
\le \frac{1}{m} 
\end{aligned}
\end{equation}
where in the last estimate we neglected the negative terms recalling the positivity of the discrete solutions $u^N $ and $v^N$. On the other hand, the integrand in the left-hand side of the above display is close to $4 m^{-2 } ( m-1) | \nabla^N (u^N)^{m/2} |^2 $ as $N$ tends to infinity. Indeed, for any $\alpha \in \mathbb{R}$, according to the mean value theorem for the function $z \mapsto z^\alpha$, $z \in \mathbb{R}$, for every $x \in \mathbb{T}^d_N$ there exists $\hat{u}^N_j (x)$ between $u^N (x)$ and $u^N (x + e_j)$ such that
\[
\nabla^N u^N (x)^\alpha - \alpha  u^N (x)^{\alpha -1 } \nabla^N u^N ( x ) 
= \big(  \alpha \partial^N_j u^N (x) ( \hat{u}^N_j (x)^{\alpha -1}  - u^N (x )^{\alpha - 1 } )  \big)_{  j =1,...,d  }. 
\]
In view of Lemma \ref{grad u}, the right-hand side has order $O ( K / N ) $ and goes to zero as $N$ tends to infinity. By this line, we have that $ \nabla^N (u^N )^{ m - 1 } \cdot  \nabla^N u^N  $ is close to $4 m^{-2 } ( m-1) | \nabla^N (u^N)^{m/2} |^2$ in the sense of $L^1 (Q_T)$-norm as $N$ tends to infinity. Therefore, according to the estimate \eqref{IBP}, we obtain 
\[
\sup_{N \in \mathbb{N}} \iint_{Q_T} | \nabla^N ( u^N )^{m/2} |^2 d\theta dt < \infty 
\] 
so that $u^{m/2} \in L^2 (0, T; H^1 (\mathbb{T}^d) )$. The weak convergence (\ref{case3 est3}) is obvious since $v^N (t,\theta) $ takes value in $[0,1]$ for all $t \in [0,T]$ and $\theta \in \mathbb{T}^d$. 

Next we show (\ref{case3 est4}). Since the $H^1(Q_T)$-norm of $u^N$ is bounded uniformly in $N$, repeating the same argument conducted in preceding two sections, the reaction term in the second equation of (\ref{case3}) can be rearranged to the spatially homogeneous one. Namely, we have  
\[
\partial_t v^N (t,x) = - K u^N (t,x)^m v^N(t,x) + O  ( K^2 / N )
.\]
After multiplying $u^N(t,x)^{m-1}$ to both sides of the first equation of (\ref{case3}), we subtract the second equation from the first one to cancel the divergent reaction term. Then, after hitting any test function $\varphi \in H^1_0 (Q_T)$, integration by parts enables us to obtain
\begin{equation}
\label{dcase3 weak}
\begin{aligned}
\iint_{Q_T} \left[ - \left( \frac{ ( u^N )^m}{m} - v^N  \right) \varphi_t  + \frac{2}{m} (u^N)^{\frac{m}{2}}  \nabla^N  (u^N)^\frac{m}{2} \cdot \nabla^N \varphi \right] d\theta dt  \\
+ \iint_{Q_T} \frac{ 4 (m-1) }{m^2}  | \nabla^N  (u^N)^{ m / 2 }  |^2 \varphi  d\theta dt 
+  O ( K^2 / N ) 
= 0 
\end{aligned}
\end{equation}
where we used a chain rule for discrete gradient: for every $\alpha \in \mathbb{R}$ and $x \in \mathbb{T}^d_N $ 
\[
| \nabla^N u^N (x)^\alpha - \alpha u^N (x)^{\alpha - 1} \nabla^N u^N (x )  | = O (K / N )
\]
as $N$ tends to infinity, which has already been proved in the above. Therefore there exists a positive constant $C$ such that
\[
\iint_{Q_T} | \nabla^N u^N  |^2 \varphi d\theta dt \leq C \| \varphi \|_{H^1_0 (Q_T)} + O(  K^2 /N  ) 
\]
for every $\varphi \in H^1_0 ( Q_T )$, which implies 
\[
\sup_{N \in \mathbb{N}} \|  | \nabla^N (u^N)^{m/2} |^2   \|_{H^{-1} (Q_T)}  < \infty
\]
and thus we end proof of (\ref{case3 est4}). Since any weak$^*$ limits of the sequence $\{ | \nabla^N ( u^N )^{m/2}  |^2  \}_{N \in \mathbb{N}}$ stay non-negative, combining the general estimates stated in Section 4, all properties in (\ref{case3 est1}) clearly hold. 

Finally, taking limit in (\ref{dcase3 weak}) along the common subsequence $(N_k)$ we obtained above, it follows that $u,v$ and $\zeta$ satisfy the weak form (\ref{case3 weak}) and hence we complete the proof. 
\end{proof}

The weak form (\ref{case3 weak}) is the same as the one which was derived in \cite{IMMN17}. They identify $\zeta$ as $| \nabla u^{m/2} |^2$ and further characterize behavior of the limiting interface as follows.

\begin{proposition}[\cite{IMMN17}]
\label{immovable}
Assume the same conditions as in Theorem \ref{case3 thm}. Let $u,v$ and $\zeta$ be the functions given in Theorem \ref{case3 thm} and assume there exists a positive constant $m_v$ such that $v(0, \cdot) \geq m_v$ in ${\rm supp } (v(0, \cdot) )$. Suppose that $\Gamma (t) $ is a smooth, closed and orientable hypersurface in $\mathbb{T}^d$ and that $\Gamma (t)$ smoothly moves with a normal speed $V$ from $\Omega^u(t)$ to $\Omega^v(t)$. Moreover, suppose that $u$ (resp. $v$) is smooth in $\overline{ Q^u_T }$ (resp. $\overline{Q^v_T} $) and that $\zeta \in L^1_{\text{loc}} (Q_T)$. Then $u, v$ and $\zeta$ satisfy the followings:
\[
\begin{aligned}
&V \equiv 0  &&\text{ on } \Gamma \\
&\partial_t u = \Delta u && \text{ in } (0, T ] \times \Omega^u (0)  \\
&u=0 && \text{ on } (0, T] \times \Gamma(0)  \\
&v=v_0 ,\quad  \zeta = | \nabla u^{m/2} |^2 && \text{ in } Q_T  
\end{aligned}
\]
\end{proposition}

\section*{Acknowledgments}
The author would like to thank Tadahisa Funaki, Hirokazu Ninomiya, Makiko Sasada and Kenkichi Tsunoda for giving him fruitful comments.


\end{document}